\documentclass{article}
\usepackage{amsthm}
\usepackage{amsmath}
\usepackage{amsfonts}
\usepackage{amssymb}
\usepackage{epsfig}
\usepackage{version}

\theoremstyle{plain}
\newtheorem{theorem}{Theorem}[section]
\newtheorem{lemma}[theorem]{Lemma}
\newtheorem{proposition}[theorem]{Proposition}

\theoremstyle{definition}

\theoremstyle{remark}
\newtheorem{remark}[theorem]{Remark}

\title{The Classification of Epimorphisms from $A(\widetilde{A}_{n})$
to $W(\widetilde{A}_{n})$}
\author{Nuno Franco \\ \small
CIMA-UE and Universidade de \'{E}vora, Departamento de Matem\'{a}tica \\ \small  R. Rom\~{a}o Ramalho, 59, 7000 \'{E}vora, Portugal \\ \small  nmf@uevora.pt}

\begin{document}

\maketitle

\begin{abstract}
We present a complete classification for the epimorphisms from the affine Artin-Tits groups $A(\widetilde{A}_{n})$
to the corresponding Coxeter groups $W(\widetilde{A}_{n})$, for $n\geq 1$.
\end{abstract}

\bigskip

\noindent MSC classes: 20F36

\section{Introdution}

In 1947 Artin, in \cite{artin1}, raised the
problem of
determining all automorphisms of the braid groups and in another paper
\cite{artin}
in the same year he began a solution,  part of which involved
determining
epimorphisms to the quotient Coxeter group. An important point he
established
was that the kernel of the mapping to the Coxeter group is
characteristic.

More recently this program has been extended to some Artin-Tits groups
in the
papers of Cohen and Paris \cite{pariscohen} and Franco and Paris \cite{francoparis}. Both papers
determine epimorphisms up to a suitable equivalence
and prove the relevant kernel is characteristic.

So a complete classification of the epimorphisms, of an Artin-Tits group to its Coxeter group, is a first step in order to study the kernel of the standard epimorphism. Our contribution, with this classification, is start the extension of Artin's program to the affine type Artin-Tits groups.

Artin-Tits groups are a widely studied class of groups. We will study  a particular sub-family of this groups the $A(\widetilde{A}_{n})$ type Artin-Tits groups. In order to do so we start by some basic definitions in order to be able to explain our main goal.

Define a \textit{Coxeter matrix} of rank $n$ as a $n\times n$
symmetric matrix, $M=(m_{i,j})$, that verifies: $m_{i,i}=1$ for
all $i=1,\ldots,n,$ and $m_{i,j}\in \{2,3,\ldots\}\cup\{\infty\}$,
for all $i,j\in\{1,\ldots,n\}$, $i\neq j$. Let $M$ be a Coxeter
matrix. The \textit{Coxeter graph} associated to $M$ is the
labelled graph, $\Gamma$, defined as follows. The set of vertices
of $\Gamma$ is $\{1,\ldots,n\}$. If $m_{i,j}=2$ then there is no
edge between $i$ and $j$, if $m_{i,j}=3,$ then there is a
non-labelled edge between $i$ and $j$, and, finally, if
$m_{i,j}>3$ or $m_{i,j}=\infty$, then there is an edge between $i$
and $j$ labelled by $m_{i,j}$.

Let $G$ be a group. For $a,b\in G$ and $n\in \mathbb{N}$ we define
the word
$$prod_n(a,b)\left\{
\begin{array}{ll}
(ab)^{\frac{n}{2}} &\text{ if }n\equiv0(\text{mod }2) \\
(ab)^{\frac{n-1}{2}}a &\text{ if } n\equiv1(\text{mod }2)
\end{array}
\right.$$ Define the \textit{Coxeter group} $W$ associated to the
Coxeter graph  $\Gamma$ as the group presented by
$$W=W(\Gamma)=\left\langle
\begin{array}{cc}
s _{1},s_{2},\ldots ,s_{n} & \left|
\begin{array}{l}
s_i^2=1,\quad i\in \{1,\ldots,n\} \\
prod_{m_{i,j}}(s_{i},s_{j})=prod_{m_{j,i}}(s_{j},s_{i}) \\
\quad \quad \quad \quad \quad \quad \quad i\neq j\text{ and
}m_{i,j}\neq\infty
\end{array}
\right.
\end{array}
\right\rangle,$$ where $M=(m_{i,j})$ is the Coxeter matrix of
$\Gamma$. Define the \textit{Artin-Tits group} associated to $\Gamma$
to be the group presented by
$$A=A(\Gamma)=\left\langle
\begin{array}{cc}
s _{1},s_{2},\ldots ,s_{n} & \left|
\begin{array}{l}
prod_{m_{i,j}}(s_{i},s_{j})=prod_{m_{j,i}}(s_{j},s_{i}) \\
\quad \quad \quad \quad \quad \quad \quad \quad i\neq j\text{ and
}m_{i,j}\neq\infty
\end{array}
\right.
\end{array}
\right\rangle.$$

If the group $W(\Gamma)$ is finite, then we say that $A(\Gamma)$
is an \textit{Artin-Tits group of spherical type}.

The Artin-Tits group of type of spherical type $A(A_{n-1})$ is defined by the Coxeter graph in figure \ref{Fig1}.
\begin{figure}
  \begin{center}
  \includegraphics{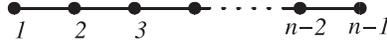}\\
  \end{center}
  \caption{Coxeter graph $A_{n-1}$}\label{Fig1}
\end{figure}These groups are also known as Artin braid groups and $W(A_{n-1})$ is isomorphic to the symmetric group $S_n$.
Let us recall the Artin presentation (see \cite{artin1}) of the braid group $A(A_{n-1})$, for $n>2$:

$$A(A_{n-1})=\left\langle
\begin{array}{cc}
\sigma _{1},\sigma _{2},\ldots ,\sigma _{n-1} & \left|
\begin{array}{ll}
\sigma _{i}\sigma _{j}=\sigma _{j}\sigma _{i} & (|i-j|\geq 2) \\
\sigma _{i}\sigma _{i+1}\sigma _{i}=\sigma _{i+1}\sigma _{i}\sigma
_{i+1} & (i=1,\ldots ,n-2)
\end{array}
\right.
\end{array}
\right\rangle. \text{ }
$$

Recall (see \cite{birman}) that the Garside element of $A(A_{n-1})$ is  $\Delta=\sigma_1\cdots\sigma_{n-1}\sigma_1\cdots\sigma_{n-2}\cdots\sigma_1\sigma_2\sigma_1.$ A positive braid is a braid which can be written  only with positive powers of the generators. A simple braid is a positive braid which divides $\Delta$.

The affine group $A(\widetilde{A}_{n-1})$, for $n>2$, is the defined by the Coxeter graph in figure \ref{Fig2}.

\begin{figure}[ht]
  \begin{center}
  \includegraphics{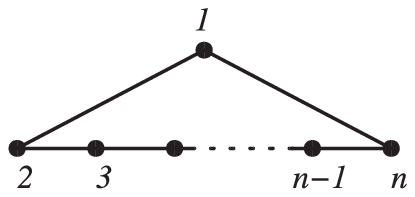}\\
  \end{center}
  \caption{Coxeter graph $\tilde{A}_{n-1}$}\label{Fig2}
\end{figure}

This is not a spherical type Artin-Tits group.
We can obtain a presentation of the affine Artin-Tits group
$A(\widetilde{A}_{n-1})$, we just add to the Artin presentation $A(A_{n-1})$
a generator $\sigma_n$ and the relations:

$$\sigma_n\sigma_1\sigma_n=\sigma_1\sigma_n\sigma_1$$
$$\sigma_n\sigma_{n-1}\sigma_n=\sigma_{n-1}\sigma_n\sigma_{n-1}$$
$$\sigma_n\sigma_i=\sigma_i\sigma_n\text{ for }i=2,\ldots,n-2.$$

The group $A(\widetilde{A}_{1})$ is defined by the Coxeter graph in figure \ref{Fig3}. This is a free group on two generators and  we will deal with it separately in section 9.

\begin{figure}[ht]
  \begin{center}
  \includegraphics{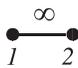}\\
  \end{center}
  \caption{Coxeter graph $\tilde{A}_{1}$}\label{Fig3}
\end{figure}

For a given Coxeter graph $\Gamma$, associated to a $n\times n$ Coxeter matrix,  there is an natural epimorphism $\mu:A(\Gamma)\rightarrow W(\Gamma)$ defined by $\mu(s_i)=s_i$ for $i=1,\ldots,n$.

\bigskip

\textbf{Main goal:} We will be interested in all the other epimorphisms in the case of the affine group of type $A(\widetilde{A}_{n-1})$. We will present a complete classification of these epimorphisms up to automorphisms of $W(\widetilde{A}_{n-1})$.

\bigskip

A full classification of the epimorphisms from the spherical type Artin-Tits groups into the corresponding Coxeter group can be found in \cite{artin}. This is the base point for several other results (see for example \cite{pariscohen}) concerning the classification of epimorphism of the Artin-Tits type groups to the Coxeter groups associated. The idea of this work is to split the Coxeter group, $W(\widetilde{A}_{n-1})$,  associated to $A(\widetilde{A}_{n-1})$. We know that $W(\widetilde{A}_{n-1})$ is isomorphic to $\mathbb{Z}^{n-1}\rtimes S_n$ (see \cite{bourbaki}) and we will project the image of an epimorphism to $S_n$ or $\mathbb{Z}^n$ in order to understand it. We will start by extending the result in \cite{artin} to the case of morphisms from the affine group $A(\widetilde{A}_{n-1})$ to $W(A_{n-1})\cong S_n$, $n>2$. In section 3, we start to reduce the number of candidates to be an epimorphism by using the relations in $A(\widetilde{A}_{n-1})$, $n>2$. In section 4, we will use another criterium on $\mathbb{Z}^n$, $n>2$ to  reduce even more our list. In section 5, we present a characterization of the kernel of a specific morphism. We proceed in section 6 by computing, up to some of the automorphisms of  $W(\widetilde{A}_{n-1})$, the images of the generators of $A(\widetilde{A}_{n-1})$, $n>2$. Using the  characterization obtained in section 5 and the results in section 6, we  write, in section 7, a complete list of epimorphisms from $A(\widetilde{A}_{n-1})$ to $W(\widetilde{A}_{n-1})$, $n>2$. We compute the classes of epimorphisms, up to automorphisms of $W(\widetilde{A}_{n-1})$, $n>2$, in section 8. We present the special case of $A(\widetilde{A}_{1})$ in section 9. In section 10 we present our main result,  a complete classification for the epimorphisms from $A(\widetilde{A}_{n-1})$
to $W(\widetilde{A}_{n-1})$, $n>1$.

\section{Epimorphisms from
$A(\widetilde{A}_{n-1})$ to de symmetric group $S_n$.}

We will present, in this section, the list of epimorphisms from
$A(\widetilde{A}_{n-1})$ to de symmetric group $S_n$, $n>2$. To do so we
start by recalling the result presented in \cite{artin} and after
it we will see as it proof can be adapted to obtain our list
epimorphisms from $A(\widetilde{A}_{n-1})$ to $S_n$.

\begin{theorem}\cite{artin} The possible representations of $A(A_{n-1})$
in the symetric group $S_n$ are:

\begin{enumerate}

\item $\sigma_i=(i,i+1)$.

\item $\sigma_1=(1,2)(3,4)(5,6)$, $D=(1,2,3)(4,5)$ for $n=6$.

\item $\sigma_1=(1,2,3,4)$, $D=(1,2)$  for $n=4$.

\item $\sigma_1=(1,3,2,4)$, $D=(1,2,3,4)$ for $n=4$.

\end{enumerate}

\end{theorem}

We present now our result, which is a direct consequence of Artin's result:

\begin{theorem}\label{thoArtin}The possible representations of $A(\widetilde{A}_{n-1})$
in the symetric group $S_n$ are:

\begin{enumerate}

\item $\sigma_i=(i,i+1)$ and $\sigma_n=(1,n)$.

\item $\sigma_1=(1,2)(3,4)(5,6)$, $D=(1,2,3)(4,5)$ and
$\sigma_6=(1,5)(2,3)(4,6)$ for $n=6$.

\item $\sigma_1=(1,2,3,4)$, $D=(1,2)$ and $\sigma_4=(1,2,4,3)$ for $n=4$.

\item $\sigma_1=(1,2,3,4)$, $D=(1,2)$ and $\sigma_4=(1,3,4,2)$ for $n=4$.

\item $\sigma_i=(i,i+1)$ and $\sigma_4=\sigma_2$, for $n=4$.

\item $\sigma_1=(1,3,2,4)$, $D=(1,2,3,4)$ and $\sigma_4=(1,3,4,2)$ for $n=4$.

\item $\sigma_1=(1,3,2,4)$, $D=(1,2,3,4)$ and $\sigma_4=(1,2,4,3)$ for $n=4$.

\end{enumerate}

\end{theorem}

\begin{proof}

The result follows from Artin's proof by slightly changing the preliminary lemmas. This is just done by assuming the existence of an extra generator and the relations containing it.

As in Artin's paper we define for $k>i$:

$$D_{ki}=\sigma_k \cdots \sigma_n\sigma_1\cdots\sigma_{i-1}.$$

and we also have as in (6):

$$D^{n-k+i+1}_{ki}=(\sigma_kD_{ki})^{n-k+i}.$$

Regarding lemma 1 from \cite{artin}, we must add or if $\sigma_1$ is commutative with $\sigma_n$. The proof remains unchanged.

Lemma 2, 4 and 6 will be use has they are, without changes.

As to Lemma 3 we must add the case $k>i$ and then we must replace the integers
$t$ by $n-t$.  The proof is obtained by rewriting, the original proof,  with a shifting of the indices.

The conclusion of Lemma 5 is the same and in the proof we just have to
put an upper limit on $i\geq 3$ and write $n>i\geq 3$.

Now the proof of the theorem uses the adaptation of Artin's lemmas in the same way. We just have to search, among all possible permutations, for the  images of the $n-$th generator for the $n=4,6$ special cases.

\end{proof}

\section{Reducing the possible epimorphisms}

As we said in the introduction, the Coxeter group $W(\widetilde{A}_{n-1})$ is isomorphic to
$\mathbb{Z}^{n-1} \rtimes S_n$. For better computations we use the
following semidirect product $\mathbb{Z}^n \rtimes S_n$, where
$S_n $ acts on $\mathbb{Z}^n$ by permuting its coordinates but we
will impose that the sum of all coordinates is 0. Typically an
element in $W(\widetilde{A}_{n-1})$ is represented by
$[x_1,\ldots,x_n]s$, where $x_i\in \mathbb{Z}$,$\sum x_i=0$ and
$s\in S_n$. If $s=id$ then we will  write just $[x_1,\ldots,x_n]$. 

We will assume that:

$$\xi(\sigma_i)=[u_{i1},\ldots,u_{in}]\xi_l(\sigma_i),$$

where $\xi_l$ corresponds to the epimorphism in the $l^{th}$ case
in theorem \ref{thoArtin}.

\subsection{Case $l=1$}\label{sec31}

\begin{proposition}\label{prop31}

Let $\sigma_i$ be the generators of $A(\widetilde{A}_{n-1})$. We
define $\xi_{(x_1,\ldots,x_n,y)}$ from $A(\widetilde{A}_{n-1})$ to
$W(\widetilde{A}_{n-1})$ as:
$$\xi_{(x_1,\ldots,x_n,y)}(\sigma_i)=[\underbrace{y,\ldots,y}_{i-1},-(n-2)y-x_i,x_i,y,\ldots,y](i,i+1), n>i\geq 1$$
$$\xi_{(x_1,\ldots,x_n,y)}(\sigma_n)=[x_n,y,\ldots,y,-(n-2)y-x_n](1,n).$$

In this section we will denote $\xi_{(x_1,\ldots,x_n,y)}$ simply by $\xi$.
For all $x_i\in \mathbb{Z}$, $\xi(\sigma_i)$ verify the braid
relations.

\end{proposition}

\begin{proof}

We will start by verifying the type $3$ relations:

\begin{enumerate}

\item Case $(i,i+1)$ with $i+1\leq n$ and $i\geq 1$.

We have

$\xi(\sigma_i) \xi( \sigma_{i+1})
\xi(\sigma_i)=\\
=[y,\ldots,y,-(n-2)y-x_i,x_i,y,\ldots,y](i,i+1)
[y,\ldots,y,-(n-2)y-x_{i+1},x_{i+1},y,\ldots,y](i+1,i+2)
\xi(\sigma_i)=\\
=[2y,\ldots,2y,-(n-2)y-x_i+y,y-(n-2)y-x_{i+1}
,x_i+x_{i+1},2y,\ldots,2y](i,i+2,i+1)
[y,\ldots,y,-(n-2)y-x_i,x_i,y,\ldots,y](i,i+1)=\\
=[3y,\ldots,3y,-2(n-2)y-x_i-x_{i+1}+y,y-(n-2)y
,y+x_i+x_{i+1},3y,\ldots,3y](i,i+2).$

and

$\xi(\sigma_{i+1}) \xi( \sigma_i) \xi(\sigma_{i+1})=\\
=[y,\ldots,y,-(n-2)y-x_{i+1},x_{i+1},y,\ldots,y](i+1,i+2)
[y,\ldots,y,-(n-2)y-x_i,x_i,y,\ldots,y](i,i+1)
\xi(\sigma_{i+1})=\\
=[2y,\ldots,2y,-2(n-2)y-x_i-x_{i+1},y+x_i
,x_{i+1}+y,2y,\ldots,2y](i,i+1,i+2)
[y,\ldots,y,-(n-2)y-x_{i+1},x_{i+1},y,\ldots,y](i+1,i+2)=\\
=\xi(\sigma_i) \xi( \sigma_{i+1}) \xi(\sigma_i)$

\item Case $(1,n)$.

We have

$\xi(\sigma_{1}) \xi( \sigma_n) \xi(\sigma_{1})=[-(n-2)y-x_1,x_1,y,\ldots,y](1,2)[x_n,y,\ldots,y,-(n-2)y-x_n](1,n)\xi(\sigma_{1})=\\
=[-(n-2)y-x_1+y,x_1+x_n,2y,\ldots,2y,y-(n-2)y-x_n](1,n,2)[-(n-2)y-x_1,x_1,y,\ldots,y](1,2)=[-(n-2)y+y,x_1+x_n+y,3y,\ldots,3y,y-2(n-2)y-x_n-x_1](2,n)$

and

$\xi(\sigma_n) \xi( \sigma_{1}) \xi(\sigma_n)=[x_n,y,\ldots,y,-(n-2)y-x_n](1,n)[-(n-2)y-x_1,x_1,y,\ldots,y](1,2)\xi(\sigma_n)=\\=[x_n+y,y+x_1,2y,\ldots,2y,-2(n-2)y-x_1-x_n](1,2,n)[x_n,y,\ldots,y,-(n-2)y-x_n](1,n)=\xi(\sigma_{1}) \xi( \sigma_n) \xi(\sigma_{1})$

\end{enumerate}

Verify type $2$ relations is much easier because it suffices to
notice that: in each product of two generators the permutations
will only act on coordinates that are equal (to $y$).

\end{proof}

\begin{proposition}
It does not exist any other solution to the braid equations.
\end{proposition}

\begin{proof}

Suppose that we have another solution for the braid equations. Let
us show first that:
$$u_{ik}=u_{i(k+1)}\text{ for }k\neq i,i+1.$$

We just have to see that from the relation

$$\sigma_k\sigma_i=\sigma_i\sigma_k$$

we obtain, among others, the equation

$$u_{kk}+u_{ik}=u_{kk}+u_{i(k+1)}.$$

We can conclude that $u_{ik}=u_{ij}$ where $k,j\neq i,i+1$.

Now we must show that:

$$u_{ik}=u_{(i+1)j}\text{ for }k\neq i,i+1\text{ and }j\neq i+1,i+2.$$

For this we will use the type $3$ braid relation

$$\sigma_i\sigma_{i+1}\sigma_i=\sigma_{i+1}\sigma_i\sigma_{i+1}$$

from which we get the equation

$$2u_{ik}+u_{(i+1)k}=u_{ik}+2u_{(i+1)k}.$$

So the new solution can differ, from the previous ones, for
each generator $\sigma_i$ only in coordinates $i,i+1$ if $i<n$ or
$1,n$ if $i=n$. But we have only one coordinate to change because
the equations
$$\sum_j u_{ij}=0$$
fix the second one. We obtain exactly previous solutions.

\end{proof}

\subsection{Case $l=2$}

In this case we are in $A(\widetilde{A}_{5})$, and we obtain, from the braid equations, the
following solutions for $\xi$:

$\xi(\sigma_1)=[x_{1}-x_{2}+x_{5}-x_{4}, -x_{1}+x_{2}-x_{5}+x_{4},
-x_{1}+x_{3},x_{1}-x_{3}, -x_{2}, x_{2}](1, 2) (3, 4) (5, 6) \\
\xi(\sigma_2)=[x_{1}-x_{2}+x_{5}-x_{3}-x_{4}, -x_{5}-x_{1}+x_{2},
-x_{1}+x_{2}-x_{5}+x_{3}+x_{4}, -x_{5}, x_{5}+x_{1}-x_{2}, x_{5}](1, 3) (2, 5) (4, 6), \\
\xi(\sigma_3)=[-x_{2}+x_{5}-x_{4}, -x_{3}, x_{3},
x_{2}-x_{5}+x_{4}, -x_{2}, x_{2}](1, 4)
(2, 3) (5,6), \\
\xi(\sigma_4)=[x_{1}-x_{2}+x_{5}-x_{4}, -x_{1}+x_{2}-x_{5}+x_{4},
-x_{1}+x_{2}-x_{5}+x_{3}, -x_{5}, x_{1}-x_{2}+x_{5}-x_{3},
x_{5}](1, 2)(3, 5) (4, 6),
\\ \xi(\sigma_5)=[x_{1}-x_{2}+x_{5}-x_{3}-x_{4}, -x_{1}, -x_{1}+x_{2}-x_{5}+x_{3}+x_{4},
x_{1}, -x_{2}, x_{2}](1, 3) (2, 4) (5, 6), \\
\xi(\sigma_6)=[-x_{4}, -x_{3}, x_{3}, -x_{5}, x_{4}, x_{5}](1, 5)
(2, 3) (4, 6)$

with $x_i \in \mathbb{Z}.$

\subsection{Case $l=3$}

In this case we are in $A(\widetilde{A}_{3})$, and we obtain, from the braid equations, the
following solutions for $\xi$:

$\xi(\sigma_1)=[x_{3}, -x_{3}-x_{1}-x_{2}, x_{1}, x_{2}](1, 2, 3, 4), \\
\xi(\sigma_2)=[-x_{2}-x_{1}, -x_{3}, x_{1}, x_{2}+x_{3}](1, 3, 4, 2), \\
\xi(\sigma_3)=[x_{3}, -x_{3}-x_{1}-x_{2}, x_{1}, x_{2}](1, 2, 3, 4), \\
\xi(\sigma_4)=[x_{3}, -x_{2}-x_{3}, x_{2}+x_{1}, -x_{1}](1,
2,4,3)$

with $x_i \in \mathbb{Z}.$

\subsection{Case $l=4$}

In this case we are in $A(\widetilde{A}_{3})$,  and we obtain, from the braid equations, the
following solutions for $\xi$:

$\xi(\sigma_1)=[-x_{1}, -x_{2}-x_{3}, x_{2}, x_{1}+x_{3}](1, 2, 3, 4), \\
\xi(\sigma_2)=[-x_{3}-x_{1}-x_{2}, x_{1}, x_{2}, x_{3}](1, 3, 4, 2), \\
\xi(\sigma_3)=[-x_{1}, -x_{2}-x_{3}, x_{2}, x_{1}+x_{3}](1, 2, 3, 4), \\
\xi(\sigma_4)=[-x_{3}-x_{1}-x_{2},x_{1}, x_{2}, x_{3}](1, 3, 4,
2)$

with $x_i \in \mathbb{Z}.$

\subsection{Case $l=5$}

In this case we are in $A(\widetilde{A}_{3})$,  and we obtain, from the braid equations, the
following solutions for $\xi$:

$\xi(\sigma_1)=[x_{1}, -x_{1}-2x_{3}, x_{3}, x_{3}](1,2), \\
\xi(\sigma_2)=[x_{3}, x_{4},-2x_{3}-x_{4}, x_{3}](2,3), \\
\xi(\sigma_3)=[x_{3}, x_{3}, x_{2},-2x_{3}-x_{2}](3,4), \\
\xi(\sigma_4)=[x_{3}, x_{4},-2x_{3}-x_{4}, x_{3}](2,3)$

with $x_i \in \mathbb{Z}.$

\subsection{Case $l=6$}

In this case we are in $A(\widetilde{A}_{3})$, and we obtain, from the braid equations, the
following solutions for $\xi$:

$\xi(\sigma_1)=[x_{2}, x_{3}, x_{1}, -x_{2}-x_{3}-x_{1}](1, 3, 2, 4), \\
\xi(\sigma_2)=[x_{2}, -x_{2}-x_{1},x_{1}+x_{3}, -x_{3}](1, 3, 4, 2), \\
\xi(\sigma_3)=[x_{2}+x_{3}+x_{1},-x_{1}, -x_{2}, -x_{3}](1, 4, 2, 3), \\
\xi(\sigma_4)=[x_{2}, -x_{2}-x_{1},x_{1}+x_{3}, -x_{3}](1, 3,4,2)$

with $x_i \in \mathbb{Z}.$

\subsection{Case $l=7$}

In this case we are in $A(\widetilde{A}_{3})$,  and we obtain, from the braid equations, the
following solutions for $\xi$:

$\xi(\sigma_1)=[x_{2}+x_{3}+x_{1}, -x_{1}, -x_{3}, -x_{2}](1, 3, 2, 4), \\
\xi(\sigma_2)=[x_{2}+x_{3}+x_{1},-x_{1}-x_{2}, -x_{1}-x_{3}, x_{1}](1, 3, 4, 2), \\
\xi(\sigma_3)=[x_{2},x_{3}, -x_{2}-x_{3}-x_{1}, x_{1}](1, 4, 2, 3), \\
\xi(\sigma_4)=[x_{1}+x_{2}, -x_{1},-x_{2}-x_{3}-x_{1},
x_{1}+x_{3}](1, 2, 4, 3)$

with $x_i \in \mathbb{Z}.$

\section{One criterium to be an Epimorphism}

We will present sufficient condition for a morphism, from a group
$G$ to $\mathbb{Z}^n$, to be an epimorphism (for details see for \cite{ZeCarlos}).

Let us introduce first some notation: Let
$\zeta:G\rightarrow\mathbb{Z}^n$ be a morphism and
$g_1,\ldots,g_k$ a generating system of $G$. Denote by $M_{\zeta}$ the
$k\times n$ matrix formed by $\zeta(g_1),\ldots,\zeta(g_k)$, and, for $p\leq n$,
$$det_{i_1,\ldots,i_p}^{j_1,\ldots,j_p}=det \left (
\left [\begin{array}{ccc}
  \zeta(g_{i_1})_{j_1} & \cdots & \zeta(g_{i_1})_{j_p} \\
  \vdots &  & \vdots \\
  \zeta(g_{i_p})_{j_1} & \cdots & \zeta(g_{i_p})_{j_p}
\end{array} \right ] \right ).$$

We will denote by $det_{i_1,\ldots,i_p}=det_{i_1,\ldots,i_p}^{1,\ldots,p}$.

\begin{lemma}\label{lemmadet1}

In the previous conditions, let $r\leq n$. If $det_{i_1,\ldots,i_n}=1$ then
$gcd(\left \{det_{q_1,\ldots,q_r}^{j_1,\ldots,j_r}, \{q_1,\ldots,q_r\}\subset \{i_1,\ldots,i_n\}\text{ and }\{j_1,\ldots,j_r\}\subset \{1,\ldots,n\}\right \})=1$, where $gcd$ stands for \emph{greatest common divisor}.

\end{lemma}

\begin{proof}

It is sufficient to notice that, for each $r$,
$det_{i_1,\ldots,i_n}$ is a linear combination of all
 $det_{q_1,\ldots,q_r}^{j_1,\ldots,j_r}$ and if their $gcd$
$\neq 1$ then we have a contradiction.

\end{proof}

Remark that the previous lemma implies that the $gcd$ of all
elements in columns of the matrix are $1$.

\begin{proposition}\label{propcrit1}
Let
$\zeta:G\rightarrow\mathbb{Z}^n$ be a morphism,
$g_1,\ldots,g_k$ a generating system of $G$ and $r\leq n$. If there exists $\{q_1,\ldots,q_r\}\subset \{1,\ldots,k\}$ and $\{j_1,\ldots,j_r\}\subset \{1,\ldots,n\}$ such that $\#\{q_1,\ldots,q_r\}=r$ and $det_{q_1,\ldots,q_r}^{j_1,\ldots,j_r}=\pm 1$ then $rank(M_{\zeta})\geq r$.
\end{proposition}

A direct consequence of Proposition \ref{propcrit1} is:

\begin{proposition}\label{propcrit}
Let
$\zeta:G\rightarrow\mathbb{Z}^n$ be a morphism and
$g_1,\ldots,g_k$ a generating system of $G$. Then $det_{i_1,\ldots,i_n}=1$ for some $\{i_1,\ldots,i_n\}\subset \{1,\ldots,k\}$ with $\#\{i_1,\ldots,i_n\}=n$  if and
only if  $\zeta$ is an epimorphism.
\end{proposition}

\section{Characterizing a kernel}

Let $n>2$. Consider de following sequence:

$$A(\widetilde{A}_{n-1})\stackrel{\xi}{\longrightarrow} W(\widetilde{A}_{n-1})\simeq \mathbb{Z}^{n-1}
\rtimes S_n \stackrel{p}{\longrightarrow} S_n$$

We will start this section by computing $\xi(ker(p \circ \xi))$.

Consider de Cayley graph $\mathcal{G}$ of $W(\widetilde{A}_{n-1})$
and its projection $p(\mathcal{G})$. Let us denote by
$Ch(p(\mathcal{G}))$ the set of paths in $p(\mathcal{G})$ with
origin in $id$. Let $\gamma\in Ch(p(\mathcal{G}))$ be a path
described by de sequence $(l_1,\ldots, l_k)$ the labels of the
edges in $\gamma$. Note that some of the $l_i$ can be negative
going a long an edge in the inverse sense. We define a morphism
$\phi$ from $Ch(p(\mathcal{G}))$ to $A(\widetilde{A}_{n-1})$ by:

$$\phi(\gamma)=\sigma^{sg(l_1)}_{|l_1|}\cdots\sigma^{sg(l_k)}_{|l_k|}.$$

\begin{proposition}

In the previous conditions:

$$\phi(\pi_1(p(\mathcal{G}),id))=ker(p \circ \xi).$$

\end{proposition}

We will construct a particular generating set, which can be
divided in two parts. One appears naturally from the fact that
$A(A_{n-1})$ injects naturally in $A(\widetilde{A}_{n-1})$. These generators are
the generators of the pure braid group $P(A(A_{n-1}))$, the kernel of the standard epimorphism $\mu$. This also allow
us to build a maximal tree of $p(\mathcal{G})$ based on $id$ which
do not contain any edge labelled $n$ and no branches leaving $id$
with length greater than $\frac{n(n-1)}{2}$. The second one is
characterized by the following lemma:

\begin{lemma}\label{lemma2ndgenerators}

All generators, $g$,  of $\phi(\pi_1(p(\mathcal{G}),id))$ that
include in its writting the $n^{th}$ generator can be written as:

$$
g=s\sigma_nt^{-1}.
$$

For some simple elements $s,t\in A_{n-1}$.

\end{lemma}

\begin{proof}
It suffices to notice that all the vertices of $p(\mathcal{G})$
are the simple elements in $A_{n-1}$. We also have that a maximal
tree in the Cayley graph of $S_n$ is a maximal tree of
$p(\mathcal{G})$. So the only way to obtain a generator with the
$n^{th}$ generator is by going from the identity to some vertex
$s$ then along the edge labeled $n$ to the vertex $t$ and return
to the identity.

\end{proof}

\section{The behavior of $W(\widetilde{A}_{n-1})$ automorphisms}

In \cite{Franzsen} the author proves that the automorphisms of
$W(\widetilde{A}_{n-1})$, $n>1$, are all inner by graph. This means that all automorphisms are inner automorphisms (conjugations) or a Coxeter graph automorphism.  Using this result we will eliminate some of the parameters of the morphisms $\xi_{(x_1,\ldots,x_n,y)}$, again denoted simply by $\xi$ on this section, defined in section \ref{sec31}. We start by using the inner automorphisms.

\begin{proposition}

Let $\xi$ from $A(\widetilde{A}_{n-1})$ to $W(\widetilde{A}_{n-1})$ be the morphism defined in proposition \ref{prop31} and $n>2$.
Then, modulo inner automorphisms of $W(\widetilde{A}_{n-1})$, we assume
that $x_i$, for $i=1,\ldots,n-1$, can appear only in the expression defining $\xi(\sigma_n).$

\end{proposition}

\begin{proof}

Consider de element
$w_{i,t}=[\underbrace{0,\cdots,0}_{i-1},t,0,\cdots,0,-t]\in
W(\widetilde{A}_{n-1})$. Let $\psi_{w_{i,t}}$ be the inner
automorphism of $W(\widetilde{A}_{n-1})$ associated to $w_{i,t}$.
So we have

$\psi_{w_{i+1,t}}(\xi(\sigma_i))=[\underbrace{y,\ldots,y}_{i-1},
-(n-2)y-x_i-t,x_i+t,y,\ldots,y](i,i+1);$

$\psi_{w_{i+1,t}}(\xi(\sigma_{i+1}))=[\underbrace{y,\ldots,y}_{i},
-(n-2)y-x_{i+1}+t,x_{i+1}-t,y,\ldots,y](i+1,i+2);$

$\psi_{w_{i+1,t}}(\xi(\sigma_{n}))=[x_n+t,y,\ldots,y,-(n-2)y-x_n-t](1,n);$

$\psi_{w_{i+1,t}}(\xi(\sigma_{j}))=\xi(\sigma_{j}) \text{ for } j\neq i,i+1,n.$

So we can put all $x_i$ only in the image of $\xi(\sigma_{n})$, with $i=1,\cdots,n-1$, conjugating
recursively by the elements

$$w_{2,-x_1},w_{3,-x_1-x_2},\cdots,w_{n-1,\sum_{i=1,\cdots,n-2}-x_i}\text{ and } w_{1,\sum_{i=1,\cdots,n-1} x_i}$$

So for the first conjugation:

$\psi_{w_{2,-x_1}}(\xi(\sigma_1))=[
-(n-2)y,0,y,\ldots,y](1,2);$

$\psi_{w_{2,-x_1}}(\xi(\sigma_{2}))=[y,
-(n-2)y-x_{2}-x_1,x_{2}+x_1,y,\ldots,y](2,3);$

$\psi_{w_{2,-x_1}}(\xi(\sigma_{n}))=[x_n-x_1,y,\ldots,y,-(n-2)y-x_n+x_1](1,n);$

$\psi_{w_{2,-x_1}}(\xi(\sigma_{j}))=\xi(\sigma_{j}) \text{ for } j\neq 1,2,n.$

The second conjugation acts as follows:

$\psi_{w_{3,-x_1-x_2}}\psi_{w_{2,-x_1}}(\xi(\sigma_1))=\psi_{w_{3,-x_1-x_2}}([
-(n-2)y,0,y,\ldots,y](1,2))=\psi_{w_{2,-x_1}}(\xi(\sigma_1));$

$\psi_{w_{3,-x_1-x_2}}\psi_{w_{2,-x_1}}(\xi(\sigma_{2}))=\psi_{w_{3,-x_1-x_2}}([y,
-(n-2)y-x_{2}-x_1,x_{2}+x_1,y,\ldots,y](2,3))=$

$=[y,-(n-2)y,0,y,\ldots,y](2,3);$

$\psi_{w_{3,-x_1-x_2}}\psi_{w_{2,-x_1}}(\xi(\sigma_{3}))=\psi_{w_{3,-x_1-x_2}}(\xi(\sigma_{3}))=$

$=\psi_{w_{3,-x_1-x_2}}([y,y,-(n-2)y-x_3,x_3,y,\ldots,y](3,4))=$

$=[y,y,-(n-2)y-x_3-x_2-x_1,x_3+x_2+x_1,y,\ldots,y](3,4);$

$\psi_{w_{3,-x_1-x_2}}\psi_{w_{2,-x_1}}(\xi(\sigma_{n}))=[x_n-2x_1-x_2,y,\ldots,y,-(n-2)y-x_n+2x_1+x_2](1,n);$

$\psi_{w_{3,-x_1-x_2}}\psi_{w_{2,-x_1}}(\xi(\sigma_{j}))=\xi(\sigma_{j}) \text{ for } j\neq 1,2,3,n.$

and so on.

The case of the last inner automorphism is slightly different:

$\psi_{w_{1,\sum_{i=1,\cdots,n-1}x_i}}\psi_{w_{n-1,\sum_{i=1,\cdots,n-2}-x_i}}
\cdots\psi_{w_{3,-x_1-x_2}}\psi_{w_{2,-x_1}}(\xi(\sigma_{n-1}))=$

$=\psi_{w_{1,\sum_{i=1,\cdots,n-1}x_i}}([y,
\ldots,y,-(n-2)y-x_{n-1}-x_{n-2}-\cdots-x_1,x_{n-1}+x_{n-2}+\cdots+x_1](n-1,n))=$

$=[x_{n-1}+x_{n-2}+\cdots+x_1,y,\ldots,y,-x_{n-1}-x_{n-2}-\cdots-x_1]$

$[y,\ldots,y,-(n-2)y-x_{n-1}-x_{n-2}-\cdots-x_1,x_{n-1}+x_{n-2}+\cdots+x_1](n-1,n)$

$[-x_{n-1}-x_{n-2}-\cdots-x_1,y,\ldots,y,x_{n-1}+x_{n-2}+\cdots+x_1]=$

$=[y,
\ldots,y,-(n-2)y,0](n-1,n).$

In the end of this process we have all the $x_i$ appear only in the expression defining $\xi(\sigma_n).$

\end{proof}

\begin{proposition}\label{propxiyp}

Let $\xi$ from $A(\widetilde{A}_{n-1})$ to $W(\widetilde{A}_{n-1})$ be the morphism defined in proposition \ref{prop31} and $n>2$.
Let $y,p\in\mathbb{Z}$, $\xi_{(y,p)}$ from $A(\widetilde{A}_{n-1})$ to $W(\widetilde{A}_{n-1})$ be the morphism defined by

$\xi_{(y,p)}(\sigma_i)=[\underbrace{y,\ldots,y}_{i-1},-(n-2)y,0,y,\ldots,y](i,i+1);$

$\xi_{(y,p)}(\sigma_n)=[p,y,\ldots,y,-(n-2)y-p](1,n);$

Then, for $p=x_n-(n-5)x_1-(n-6)x_2-\cdots-x_{n-4}+2x_{n-1}+x_{n-2}$, $\xi_{(y,p)}$ and $\xi$  are equal up to automorphisms of $W(\widetilde{A}_{n-1})$.

\end{proposition}

\begin{proof}

After conjugating $\xi$ by $\psi=\psi_{w_{1,\sum_{i=1,\cdots,n-1} x_i}}\psi_{w_{n-1,\sum_{i=1,\cdots,n-2}-x_i}}\cdots\psi_{w_{3,-x_1-x_2}}\psi_{w_{2,-x_1}},$ we obtain:

$\psi(\xi(\sigma_i))=[\underbrace{y,\ldots,y}_{i-1},-(n-2)y,0,y,\ldots,y](i,i+1)=\xi_{(y,p)}(\sigma_i),\text{ for } 1 \leq i <n.$

To see what happens to  $\psi(\xi(\sigma_n))$ we will compute the last conjugation:

$\psi_{w_{1,\sum_{i=1,\cdots,n-1}x_i}}\psi_{w_{n-1,\sum_{i=1,\cdots,n-2}-x_i}}\cdots\psi_{w_{3,-x_1-x_2}}\psi_{w_{2,-x_1}}(\xi(\sigma_n))=$

$=\psi_{w_{1,\sum_{i=1,\cdots,n-1}x_i}}([x_n-(n-3)x_1-(n-4)x_2-\cdots-2x_{n-3}-x_{n-2},$

$,y,\ldots,y,-(n-2)y-x_n+(n-3)x_1+(n-4)x_2-\cdots+2x_{n-3}+x_{n-2}](1,n))=$

$=[x_{n-1}+x_{n-2}+\cdots+x_1,y,\ldots,y,-x_{n-1}-x_{n-2}-\cdots-x_1]$

$[x_n-(n-3)x_1-(n-4)x_2-\cdots-2x_{n-3}-x_{n-2},y,\ldots,y,-(n-2)y-x_n+(n-3)x_1+(n-4)x_2-\cdots+2x_{n-3}+x_{n-2}](1,n)$

$[-x_{n-1}-x_{n-2}-\cdots-x_1,y,\ldots,y,x_{n-1}+x_{n-2}+\cdots+x_1]=$

$=[x_n-(n-4)x_1-(n-5)x_2-\cdots-x_{n-3}+x_{n-1},y,
\ldots,y,-(n-2)y-x_n+(n-4)x_1+(n-5)x_2+\cdots+x_{n-3}-x_{n-1}](1,n)$

$[-x_{n-1}-x_{n-2}-\cdots-x_1,y,\ldots,y,x_{n-1}+x_{n-2}+\cdots+x_1]=$

$=[x_n-(n-5)x_1-(n-6)x_2-\cdots-x_{n-4}+2x_{n-1}+x_{n-2},$

$,y,\ldots,y,-(n-2)y-x_n+(n-5)x_1+(n-6)x_2+\cdots+x_{n-4}-2x_{n-1}-x_{n-2}](1,n).$

So $p=x_n-(n-5)x_1-(n-6)x_2-\cdots-x_{n-4}+2x_{n-1}+x_{n-2}.$

\end{proof}

\begin{remark}

Proposition \ref{propxiyp} states that instead of $n+1$ parameters in $\xi$, the $x_i$, for $i=1,\ldots,n$, and $y$, we just have to deal with
two parameters $y$ and $p$. It does not matter the values of each individual $x_i$, it just matters the value of $p$.
So from now on we will just deal with the morphisms $\xi_{(y,p)}$ with $y,p\in \mathbb{Z}.$

\end{remark}

The graph automorphisms form a dihedral group generated by a
rotation, $\rho$, of the graph by an angle of $\frac{2\pi}{n}$ and
a symmetry.

\medskip

\textbf{The rotation $\rho$}

\medskip

In order to understand the action of $\rho$ in $\mathbb{Z}^n$ we
will use as generating set, for $\mathbb{Z}^n$,
$\{w_i=w_{i,1}\}_{i=1,\ldots,n-1}.$ Now we fix

$$s_i=(i,i+1)\text{ for }i=1,\ldots,n-1$$
and
$$s_n=[1,0,\ldots,0,-1](1,n),$$

as generators of $W(\widetilde{A}_{n-1})$. The rotation $\rho$ (see Figure \ref{Fig4}), is
defined by:

$$\rho(s_i)=s_{i+1}\text{ for }i\neq n;$$

$$\rho(s_n)=s_1.$$

\begin{figure}[ht]
  \begin{center}
  \includegraphics{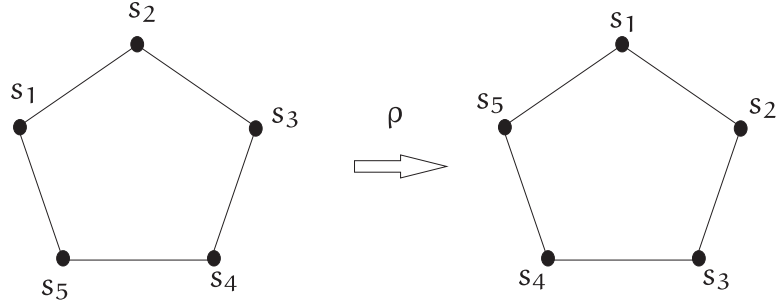}\\
  \end{center}
  \caption{The rotation $\rho$ acting on Coxeter graph $\tilde{A}_{4}$}\label{Fig4}
\end{figure}

\begin{lemma}\label{lemmaw}

Let $1 \leq k<n$. In the previous conditions we have

$$w_{1}=s_ns_1\cdots s_{n-2}s_{n-1}s_{n-2}\cdots s_1;$$

$$w_{k}=s_{k-1}\cdots s_1 w_{1} s_1\cdots s_{k-1}.$$

\end{lemma}

\begin{proof}

The case $w_{1}$ is just a direct computation using the fact that
$s_1\cdots s_{n-2}s_{n-1}s_{n-2}\cdots s_1=(1,n).$ Suppose that

$w_{k-1}=s_{k-2}\cdots s_1 w_{1} s_1\cdots
s_{k-2}=[\underbrace{0,\ldots,0}_{k-2},1,0,\ldots,0,-1].$

Now we have that

$s_{k-1}w_{k-1}s_{k-1}=s_{k-1}s_{k-1}[\underbrace{0,\ldots,0}_{k-1},1,0,\ldots,0,-1]=w_{k}.$

\end{proof}

\begin{proposition}\label{proprho1}

Let $1 \leq k<n$. In the previous conditions, we have:

$$\rho(w_{i})=w_{i+1}w_{1}^{-1}\text{ for }i=1,\ldots,n-2;$$

and

$$\rho(w_{n-1})=w_{1}^{-1}.$$

\end{proposition}

\begin{proof}

We will start by computing $\rho(w_{i}).$ By lemma \ref{lemmaw} we
have

$$w_{1}=s_ns_1\cdots s_{n-2}s_{n-1}s_{n-2}\cdots s_1,$$

so

$\rho(w_{1})=s_1s_2\cdots s_{n-1}s_{n}s_{n-1}\cdots
s_2=(1,2)(2,3)\cdots(n-1,n)[1,0,\ldots,0,-1](1,n)\\(n-1,n)\cdots(2,3)=(n,\ldots,1)[1,0,\ldots,0,-1](1,n)(2,\ldots,n)=
(n,\ldots,1)[1,0,\ldots,0,-1](1,\ldots,n)=[-1,1,0,\ldots,0]=w_{2}w_{1}^{-1}.$

Suppose that $\rho(w_{k})=w_{k+1}w_{1}^{-1}\text{ for }1<k<n-1.$ By
lemma \ref{lemmaw} we have

$\rho(w_{k+1})=\rho(s_kw_{k}s_k)=s_{k+1}[-1,\underbrace{0,\ldots,0}_{k-1},1,0,\ldots,0]s_{k+1}=\\
=[-1,\underbrace{0,\ldots,0}_{k},1,0,\ldots,0]=w_{k+2}w_{1}^{-1}.$

Suppose that $k=n-1.$ By lemma \ref{lemmaw} we have

$\rho(w_{n-1})=\rho(s_{n-2}w_{n-2}s_{n-2})=s_{n-1}[-1,\underbrace{0,\ldots,0}_{n-3},1,0]s_{n-1}=\\
=[-1,0,\ldots,0,1]=w_{1}^{-1}.$

\end{proof}

\medskip

\textbf{The symmetry $\gamma$}

\medskip

Now we analyze the other generating automorphism, the symmetry $\gamma$ (see Figure \ref{Fig5}), of
$W(\widetilde{A}_{n-1})$. Let $s_i$, as before, denote the image
of $\sigma_i$ by the standard epimorphism. So

$$\gamma(s_1)=s_1;$$

and

$$\gamma(s_i)=s_{n-i+2}\text{ for }i\neq 1.$$

\begin{figure}[ht]
  \begin{center}
  \includegraphics{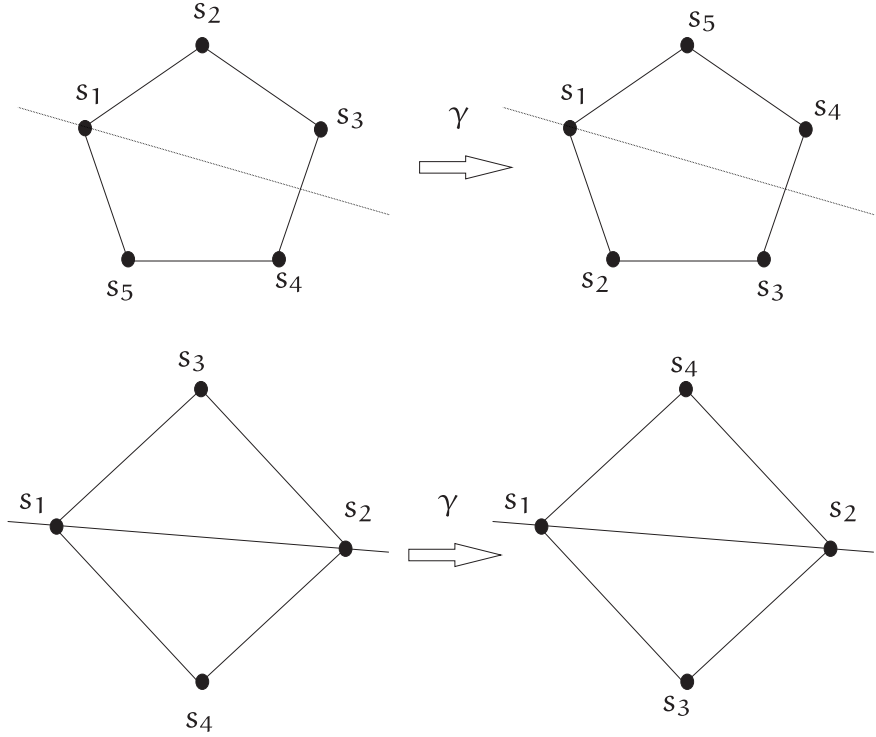}\\
  \end{center}
  \caption{The symmetry  $\gamma$ acting on Coxeter graphs $\tilde{A}_{4}$ and $\tilde{A}_{3}$}\label{Fig5}
\end{figure}

\begin{proposition}\label{propsym}

In these conditions we have:

$$\gamma(w_{1})=w_2^{-1}w_3;$$

$$\gamma(w_{2})=w_{3}w_{1}^{-1}$$

$$\gamma(w_{3})=w_{3};$$

$$\gamma(w_{i})=w_{3}w_{n+3-i}^{-1},\text{
for }i=4,\ldots,n-1;$$

\end{proposition}

\begin{proof}

We will start by computing $\gamma(w_{i}).$ By lemma \ref{lemmaw}
we have

$$w_{1}=s_ns_1\cdots s_{n-2}s_{n-1}s_{n-2}\cdots s_1,$$

so for $i=1$ we have

$\gamma(w_{1})=s_2s_1s_{n}s_{n-1}\cdots s_{4}s_{3}s_{4}\cdots
s_{n}s_1=$

$=s_2s_1[1,0,\ldots,0,-1](1,n)(3,n)[1,0,\ldots,0,-1](1,n)s_1=$

$=[0,0,1,0\ldots,0,-1]s_2s_1(1,n)(3,n)[1,0,\ldots,0,-1](1,n)s_1=$

$=[0,0,1,0\ldots,0,-1]s_2s_1(1,n)[1,0,-1,0,\ldots,0](3,n)(1,n)s_1=$

$=[0,-1,1,0\ldots,0]s_2s_1(1,n)(3,n)(1,n)s_1=w_2^{-1}w_3.$

Let $i=2.$

$$\gamma(w_{2})=\gamma(s_{1}w_{1}s_1)=\gamma(s_1)w_{2}^{-1}w_{3}\gamma(s_1)=w_{3}w_{1}^{-1}.$$

Let $i=3$ then

$\gamma(w_{3})=\gamma(s_{2}w_{2}s_2)=\gamma(s_2)w_{3}w_{1}^{-1}\gamma(s_2)=$

$=s_{n}w_{3}w_{1}^{-1}s_n=
w_1(1,n)w_{3}(1,n)=w_{3}(1,n)(1,n)=w_{3}.$

Let $i=4$ then

$\gamma(w_{4})=\gamma(s_{3}w_{3}s_3)=\gamma(s_3)w_{3}^{-1}\gamma(s_3)=$

$=s_{n-1}w_{3}^{-1}s_{n-1}=
w_{3}^{-1}w_{n-1}.$

Suppose that the formula is valid for $4<i<n.$ We have

$\gamma(w_{i+1})=\gamma(s_{i}w_{i}s_i)=\gamma(s_i)w_{3}w_{n-i+3}^{-1}\gamma(s_i)=$

$=s_{n-i+2}w_{3}
w_{n-i+3}^{-1}s_{n-i+2}= w_{3}w_{n-(i+1)+3}^{-1}.$

\end{proof}

\section{Conditions to be an Epimorphism}

In this section we will use the criterium of the previous section.
We will obtain for each case $l$ conditions on the parameters that
tell us when $\xi_l$ is an epimorphism. Let us fix $n>2$.

\subsection{Case l=1}

We are working in $A(\widetilde{A}_{n-1})$ and we have to compute
$\xi_{(y,p)}(\phi(\pi_1(p(\mathcal{G}),id))$. We will start by studding
this set.

\begin{lemma}\label{lemseq1}

Let $y,p\in\mathbb{Z}$, $\sigma_i \in A(\widetilde{A}_{n-1})$, $1\leq i \leq n-1$ we have:

$$\xi_{(y,p)}(\sigma_k\cdots\sigma_1)=[\underbrace{-(n-1-k)y,\ldots,-(n-1-k)y}_k,0,ky,\ldots,ky](1,\ldots,k+1)$$

and

$$\xi_{(y,p)}(\sigma_1\cdots\sigma_k)=[-(nk-2k)y,\underbrace{(k-1)y,\ldots,(k-1)y}_{k},ky,\ldots,ky](1,k+1,\ldots,2)$$

 for $k<n$.

\end{lemma}

\begin{proof}

By induction the case  $k=1$ being trivial let us look to the
general case.

$\xi_{(y,p)}(\sigma_{k+1}\cdots\sigma_1)=\xi_{(y,p)}(\sigma_{k+1})\xi_{(y,p)}(\sigma_{k}\cdots\sigma_1)=\\
=[y,\ldots,y,-(n-2)y,0,y,\ldots,y](k+1,k+2)[\underbrace{-(n-1-k)y,\ldots,
-(n-1-k)y}_k,0,ky,\ldots,ky](1,\ldots,k+1)=\\=[\underbrace{-(n-1-(k+1))y,
\ldots,-(n-1-(k+1))y}_{k+1},0,(k+1)y,\ldots,(k+1)y](1,\ldots,k+2)$

and

$\xi_{(y,p)}(\sigma_{1}\cdots\sigma_{k+1})=\xi_{(y,p)}(\sigma_{1}\cdots\sigma_k)\xi_{(y,p)}(\sigma_{k+1})=\\
=[-(nk-2k)y,\underbrace{(k-1)y,\ldots,(k-1)y}_k,ky,\ldots,ky](1,k+1,\ldots,2)[y,\ldots,y,-(n-2)y,0,y,\ldots,y](k+1,k+2)=\\
=[-(n(k+1)-2(k+1))y,\underbrace{ky,\ldots,ky}_{k+1},(k+1)y,\ldots,(k+1)y](1,k+2,\ldots,2).$

\end{proof}

\begin{lemma}\label{lemseq2}

Let $y,p\in\mathbb{Z}$, $\sigma_i \in A(\widetilde{A}_{n-1})$, $1\leq i \leq n-1$ we have:

$$\xi_{(y,p)}((\sigma_k\cdots\sigma_1)^{-1})=[0,\underbrace{(n-1-k)y,\ldots,(n-1-k)y}_k,-ky,\ldots,-ky](1,k+1,\ldots,2)$$

and

$$\xi_{(y,p)}((\sigma_1\cdots\sigma_k)^{-1})=[\underbrace{-(k-1)y,\ldots,-(k-1)y}_k,(nk-2k)y,-ky,\ldots,-ky](1,\ldots,k+1)$$

for $k<n$.

\end{lemma}

\begin{proof}

Notice that $\xi_{(y,p)}((\sigma_k\cdots\sigma_1)^{-1})\xi_{(y,p)}(\sigma_k\cdots\sigma_1)=\xi_{(y,p)}(\sigma_k\cdots\sigma_1)\xi_{(y,p)}((\sigma_k\cdots\sigma_1)^{-1})=[0,\ldots,0]$ and $\xi_{(y,p)}((\sigma_1\cdots\sigma_k)^{-1})\xi_{(y,p)}(\sigma_1\cdots\sigma_k)=\xi_{(y,p)}(\sigma_1\cdots\sigma_k)\xi_{(y,p)}((\sigma_1\cdots\sigma_k)^{-1})=[0,\ldots,0].$

\end{proof}

\begin{lemma}\label{lemgentype1}

Let $g$ be a generator of $P(A(A_{n-1}))$. Then
$\xi_{(y,p)}(g)=[Q_1(y),\ldots,Q_n(y)]$ for some homogeneous polynomials
$Q_i$.

\end{lemma}

\begin{proof}

This result is trivial by the preceding lemmas. We will write a
proof of it to be able to obtain more information about the
polynomials $Q_i$. It is well known (see \cite{birman}) that the elements
$$a_{ij}=\sigma_{j-1}\cdots\sigma_{i+1}\sigma_{i}^2\sigma_{i+1}^{-1}\cdots\sigma_{j-1}^{-1}$$
generate the pure braid group $P(A(A_{n-1}))$. So let us compute
$\xi_{(y,p)}(a_{ij})$ for all $i<j$:

$\xi_{(y,p)}(a_{ij})=\xi_{(y,p)}(\sigma_{j-1}\cdots\sigma_{1})(\xi_{(y,p)}(\sigma_{i}\cdots\sigma_{1}))^{-1}\xi_{(y,p)}(\sigma_{i}^2)
(\xi_{(y,p)}(\sigma_{j-1}\cdots\sigma_{1})(\xi_{(y,p)}(\sigma_{i}\cdots\sigma_{1}))^{-1})^{-1}$

So we start by computing

$\xi_{(y,p)}(\sigma_{j-1}\cdots\sigma_{1})(\xi_{(y,p)}(\sigma_{i}\cdots\sigma_{1}))^{-1}=[\underbrace{-(n-2-j)y,\ldots,-(n-2-j)y}_{j-1},
0,(j-1)y,\ldots,(j-1)y](1,\ldots,j) [0,\underbrace{(n-1-i)y,\ldots,(n-1-i)y}_i,-iy,\ldots,-iy](1,i+1,\ldots,2)=\\
=[\underbrace{(j-i+1)y,\ldots,(j-i+1)y}_i,\underbrace{-(n-2-j+i)y,\ldots,-(n-2-j+i)y}_{j-i-1},0,(j-1-i)y,\ldots,(j-1-i)y](i+1,\ldots,j).$

now we have

$(\xi_{(y,p)}(\sigma_{j-1}\cdots\sigma_{1})(\xi_{(y,p)}(\sigma_{i}\cdots\sigma_{1}))^{-1})^{-1}=[-(j-i+1)y,\ldots,-(j-i+1)y,
0,\underbrace{(n-2-j+i)y,\ldots,(n-2-j+i)y}_{j-i-1},-(j-1-i)y,\ldots,-(j-1-i)y](i+1,j,\ldots,i+2).$

For the missing factor it is easy to compute and we have:

$\xi_{(y,p)}(\sigma_i^2)=[y,\ldots,y,-(n-2)y,0,y,\ldots,y](i+1,i+2)[y,
\ldots,y,-(n-2)y,0,y,\ldots,y](i,i+1)=\\
=[2y,\ldots,2y,-(n-2)y,-(n-2)y,2y,\ldots,2y].$

So

$\xi_{(y,p)}(a_{ij})=\xi_{(y,p)}(\sigma_{j-1}\cdots\sigma_{1})(\xi_{(y,p)}(\sigma_{i}\cdots\sigma_{1}))^{-1}\xi_{(y,p)}(\sigma_{i}^2)
(\xi_{(y,p)}(\sigma_{j-1}\cdots\sigma_{1})(\xi_{(y,p)}(\sigma_{i}\cdots\sigma_{1}))^{-1})^{-1}=\\
=[\underbrace{(j-i+3)y,\ldots,(j-i+3)y}_{i-1},(j-i+3-n)y,\underbrace{-(n-4-j+i)y,\ldots,-(n-4-j+i)y}_{j-i-1},
-(n-2)y,\underbrace{(j-i+1)y,\ldots,(j-i+1)y}_{n-j+1}](i+1,\ldots,j)
(\xi_{(y,p)}(\sigma_{j-1}\cdots\sigma_{1})(\xi_{(y,p)}(\sigma_{i}\cdots\sigma_{1}))^{-1})^{-1}=\\
=[\overbrace{\underbrace{2y,\ldots,2y}_{i-1},-(n-2)y,2y,\ldots,2y}^{j-1},-(n-2)y,2y,\ldots,2y]$

\end{proof}

\begin{remark}\label{remgentype1}

We notice, from the proof of lemma \ref{lemgentype1} that not only
$\xi_{(y,p)}(g)=[Q_1(y),\ldots,Q_n(y)]$ for some homogeneous polynomials
$Q_i$, but in the case were $n$ is even, we have
$$\xi_{(y,p)}(g)=[Q'_1(2y),\ldots,Q'_n(2y)].$$

\end{remark}

\begin{lemma}\label{lemgentype2}

Let $g\in\phi(\pi_1(p(\mathcal{G}),id))$ be a generator such that

$$
g=s\sigma_nt^{-1}.
$$

For some simple elements $s,t\in A_{n-1}$. Then
$$\xi_{(y,p)}(g)=[Q_1(y,p),\ldots,Q_n(y,p)].$$
Were $Q_i$ are homogeneous polynomials with integer coefficients.

\end{lemma}

\begin{proof}

By Lemma \ref{lemma2ndgenerators} we know that it suffices to check this type of generators.
We will identify $s$ and $t$ with its images in $S_n$. Now suppose
that $l(s)\geq l(t)$, we will proceed by induction on $l(t)$.

If $l(t)=0$ then
$s=\sigma_1\cdots\sigma_{n-1}\sigma_{n-2}\cdots\sigma_1$ which is
the simple element with associated permutation $(1,n)$. We have

$\xi_{(y,p)}(g)=\xi_{(y,p)}(s)\xi_{(y,p)}(\sigma_n)=\xi_{(y,p)}(\sigma_1\cdots\sigma_{n-1})\xi_{(y,p)}(\sigma_{n-2}\cdots\sigma_1)\xi_{(y,p)}(\sigma_n)= \\
=[-(n(n-1)-2(n-1))y,(n-2)y,\ldots,(n-2)y](1,n,\ldots,2)
[\underbrace{-y,\ldots,-y}_{n-2},0,(n-2)y](1,\ldots,n-1)\xi_{(y,p)}(\sigma_n)=\\
=[-(n(n-1)-(3n-4))y,(n-3)y,\ldots,(n-3)y,(n-2)y](1,n)
[p,y,\ldots,y,-(n-2)y-p](1,n)=\\
=[-(n(n-1)-2(n-1))y-p,(n-2)y,\ldots,(n-2)y,(n-2)y+p].$

If the result is true for $l(t)=r$, suppose that $\sigma_k$ is a
initial letter of $t$, then we write:

$$g=s\sigma_nt^{-1}=s\sigma_nt_1^{-1}\sigma_k^{-1}.$$

Suppose that $\sigma_k$ is also an initial letter of $s$, then

$$g=\sigma_k s_1\sigma_nt_1^{-1}\sigma_k^{-1}.$$

Now we know that
$\xi_{(y,p)}(s_1\sigma_nt_1^{-1})=[Q_1(y,p),\ldots,Q_n(y,p)]$ by
hypothesis and so

$\xi_{(y,p)}(g)=[\underbrace{y,\ldots,y,-(n-2)y}_k,0,y,\ldots,y](k,k+1)[Q_1(y,p),\ldots,Q_n(y,p)]
[\underbrace{-y,\ldots,-y,0}_k,(n-2)y,-y,\ldots,-y](k,k+1)=\\
=[y+Q_1(y,p),\ldots,y+Q_{k-1}(y,p),-(n-2)y+Q_{k+1}(y,p),Q_{k}(y,p),
y+Q_{k+2}(y,p),\ldots,y+Q_n(y,p)](k,k+1)[-y,\ldots,-y,0,(n-2)y,-y,\ldots,-y](k,k+1)=\\
=[Q_1(y,p),\ldots,Q_{k-1}(y,p),Q_{k+1}(y,p),Q_{k}(y,p),
Q_{k+2}(y,p),\ldots,Q_n(y,p)].$

Suppose now that $\sigma_k$ is not an initial letter of $s$, then

$$g=\sigma_k^{-1}\sigma_k s_1\sigma_nt_1^{-1}\sigma_k^{-1}.$$

So we have $\sigma_k s_1$ is a simple element hence $\sigma_k
s_1\sigma_nt_1^{-1}$ is a generator and $\xi_{(y,p)}(\sigma_k
s_1\sigma_nt_1^{-1})=[Q_1(y,p),\ldots,Q_n(y,p)]$ and
$\xi_{(y,p)}(g)=[\underbrace{-y,\ldots,-y}_{k-1},0,(n-2)y,-y,\ldots,-y](k,k+1)
[Q_1(y,p),\ldots,Q_n(y,p)]
[\underbrace{-y,\ldots,-y}_{k-1},0,(n-2)y,-y,\ldots,-y](k,k+1)=\\
=[-y+Q_1(y,p),\ldots,-y+Q_{k-1}(y,p),Q_{k+1}(y,p),(n-2)y+Q_{k}(y,p),
-y+Q_{k+2}(y,p),\ldots,-y+Q_n(y,p)](k,k+1)
[\underbrace{-y,\ldots,-y}_{k-1},0,(n-2)y,-y,\ldots,-y](k,k+1)=\\
=[-2y+Q_1(y,p),\ldots,-2y+Q_{k-1}(y,p),(n-2)y+Q_{k+1}(y,p),(n-2)y+Q_{k}(y,p),
-2y+Q_{k+2}(y,p),\ldots,-2y+Q_n(y,p)].$

\end{proof}

\begin{remark}\label{remgentype2}

We notice, from the proof of lemma \ref{lemgentype2} that not only
$\xi_{(y,p)}(g)=[Q_1(y,p),\ldots,Q_n(y,p)]$ for some homogeneous
polynomials $Q_i$, but in the case were $n$ is even, we have
$$\xi_{(y,p)}(g)=[Q'_1(2y,p),\ldots,Q'_n(2y,p)].$$

\end{remark}

\begin{proposition}\label{proppoly}

Let $g\in \phi(\pi_1(p(\mathcal{G}),id))$ then
$$\xi_{(y,p)}(g)=[Q_1(y,p),\ldots,Q_n(y,p)].$$
Were $Q_i$ are homogeneous polynomials with integer coefficients.

\end{proposition}

\begin{proof}

It is obvious that we can suppose that $g$ is a generator of
$\phi(\pi_1(p(\mathcal{G}),id)$. As we saw in section 3 there are
two types of generator, those appearing as generators of
$P(A(A_{n-1}))$ and the ones in which we have the occurrence of
$\sigma_n$. The cases are both treated in lemmas \ref{lemgentype1}
and \ref{lemgentype2} respectively.

\end{proof}

Now the first result concerning the parameters is:

\begin{proposition}

If $gcd(y,p)\neq1$ then $\xi_{(y,p)}$ is not an epimorphism.

\end{proposition}

\begin{proof}

Suppose that $gcd(y,p)=q\neq1$. By proposition \ref{proppoly} we
can conclude that $Im(\xi_{(y,p)})\subset (q\mathbb{Z})^{n}$.

\end{proof}

For the rest of this section we will suppose that $gcd(y,p)=1$.

In order to use the criterium described above, it is sufficient to
verify it for a submatrix of $M$. In order to do so we will choose
some particular elements in the image of $\xi_{(y,p)}$ and show after that
this submatrix, on the first $n-1$ columns, verify the hypothesis
of the criterium.

Let us define:

$$g_k=\sigma_{k-1}\cdots\sigma_1\sigma_{k+1}\cdots\sigma_n(\sigma_{k}\cdots\sigma_1
\sigma_{k+1}\cdots\sigma_{n-1})^{-1}\text{ for }k=1,\ldots,n-1;$$
$$g_n=\sigma_n^2;$$
$$g_{n+i}=\sigma_i^2.$$

\begin{lemma}\label{lemmagk}

In the previous conditions we have:

$$\xi_{(y,p)}(g_k)=[\underbrace{0,\ldots,0}_{k-1},p+(n^2-nk-n)y,-p-(n^2-nk-n)y,0,\ldots,0].$$

\end{lemma}

\begin{proof}

This proof is a simple and direct computation using lemmas
\ref{lemseq1} and \ref{lemseq2}. Let us decompose

$$g_k=\sigma_{k-1}\cdots\sigma_1(\sigma_1\cdots\sigma_k)^{-1}(\sigma_{1}\cdots\sigma_{n-1})\sigma_n
(\sigma_{1}\cdots\sigma_{n-1})^{-1}(\sigma_{1}\cdots\sigma_{k})(\sigma_{k}\cdots\sigma_{1})^{-1}.$$

Assume for now that $k$ is even. We will start by computing

$\sigma_{k-1}\cdots\sigma_1(\sigma_1\cdots\sigma_k)^{-1}(\sigma_{1}\cdots\sigma_{n-1})\sigma_n=\\
=[\underbrace{-(n-1)y,\ldots,-(n-1)y}_{k-1},-(k-1)y,
(nk-k-1)y,-y,\ldots,-y](1,3,\ldots,k+1)(2,4,\ldots,k)
(\sigma_{1}\cdots\sigma_{n-1})\sigma_n=\\
=[\underbrace{-y,\ldots,-y}_{k-1},(n-k-1)y,
(-n^2+3n-k+nk-3)y,\underbrace{(n-3)y,\ldots,(n-3)y}_{n-k-1}]
(1,\ldots,k)(n,\ldots,k+1)\sigma_n=\\
=[\underbrace{0,\ldots,0}_{k-1},(n-k-1)y+p,
-p+(-n^2+2n-k+nk-1)y,\underbrace{(n-2)y,\ldots,(n-2)y}_{n-k-1}]
(1,\ldots,k,n,\ldots_k+1).$

Assume for now that $k$ is odd the only change is in step 1, we
must replace the permutation $(1,3,\ldots,k+1)(2,4,\ldots,k)$ by
$(1,3,\ldots,k,2,4,\ldots,k+1)$ but after we obtain exactly the
same result.

Now we compute

$(\sigma_{1}\cdots\sigma_{n-1})^{-1}(\sigma_{1}\cdots\sigma_{k})(\sigma_{k}\cdots\sigma_{1})^{-1}=\\
=[\underbrace{-(n-k-1)y,\ldots,-(n-k-1)y}_{k},\underbrace{-(n-k-2)y,\ldots,-(n-k-2)y}_{n-k-1},
(n^2-3n+2-nk-2k)y](k+1,\ldots,n)(\sigma_{k}\cdots\sigma_{1})^{-1}=\\
=[-(n-k-1)y,\underbrace{0,\ldots,0}_{k-1},-(n-2)y,\ldots,-(n-2)y,
(n^2-2n-nk+2k+1)y](k+1,\ldots,n,k,\ldots,1).$

So we have our result by multiplying these to elements in order to
obtain $\xi_{(y,p)}(g_k)$.

\end{proof}

A direct consequence of Lemma \ref{lemmagk} is:

\begin{lemma}\label{lemmagn}
For $i=1,\ldots,n-1$ we have:
$$\xi_{(y,p)}(g_{n})=[-(n-2)y,2y,\ldots,2y,-(n-2)y].$$
$$\xi_{(y,p)}(g_{n+i})=[\underbrace{2y,\ldots,2y}_{i-1},-(n-2)y,-(n-2)y,2y,\ldots,2y].$$
\end{lemma}

We must now compute some determinants in order to prove later that
their $gcd$ is $1$. Using as a list of generators
$g_1,\ldots,g_{2n-1}$ defined above we have:

\begin{lemma}\label{lemgcd1}

Let $k=1,\ldots,n-1$, then
$$gcd(y,p-(n^2-nk-n)y)=1.$$

\end{lemma}

\begin{proof}

Suppose that we have $p-(n^2-nk-n)y=qt_1$ and $y=qt_2$ with $q>1$.
Then $p=q(t_2+(n^2-nk-n)t_1)$ meaning that $gcd(p,y)\geq q$ which
is false.

\end{proof}

\begin{lemma}\label{lemdets1}
$$det_{1,\ldots,n-1}=\prod_{k=1,\ldots,n-1}(p-(n^2-kn-n)y);$$
$$det_{n+2,2,\ldots,n-1}=2y\prod_{k=2,\ldots,n-1}(p-(n^2-kn-n)y);$$
$$det_{n+1,2,\ldots,n-1}=(n-2)y\prod_{k=2,\ldots,n-1}(p-(n^2-kn-n)y).$$
\end{lemma}

\begin{proposition}

Let $y=0$ and $p=\pm 1$. Then $\xi_{(y,p)}$ is an epimorphism.

\end{proposition}

\begin{proof}

By lemma \ref{lemdets1} we have that $det_{1,\ldots,n-1}=\pm 1$.
By proposition \ref{propcrit} we are done.

\end{proof}

\begin{lemma}\label{lemdets2}

Let $n$ be odd and $y\neq 0$. There exist an integer $C\neq 0$
such that

$$det_{n+1,\ldots,2n-1}=Cy^{n-1}.$$

\end{lemma}

\begin{proof}

We will show that the matrix formed by the $\xi_{(y,p)}(g_{n+i})$ (see Lemmas \ref{lemmagn}), with
$i=1,\ldots,n-1$ as $rank(n-1)$.

\footnotesize

 $$ R=\left [\begin{matrix}
    -(n-2)y & -(n-2)y & 2y & \cdots & \cdots & \cdots & 2y \\
    2y & -(n-2)y & -(n-2)y & 2y & \cdots & \cdots & 2y \\
    \vdots &  & \ddots & \ddots &  &  & \vdots \\
    \vdots &  &  & \ddots & \ddots &  & \vdots \\
    \vdots &  &  &  & \ddots & \ddots & \vdots \\
    2y & \cdots & \cdots & \cdots & -(n-2)y & -(n-2)y & 2y \\
    2y & \cdots & \cdots & \cdots & 2y & -(n-2)y & -(n-2)y \
  \end{matrix} \right ]$$
\normalsize

 Now the first row cannot be a linear (with integer
coefficients) combination of the remaining $n-2$ rows, because
that would give us, using the first column the equation

$$-(n-2)y=2ky, \text{ for some }k\in \mathbb{Z}$$

which implies that $n-2$ is even. So now we repeat the process to
show that the second row cannot be a linear (with integer
coefficients) combination of the remaining $n-3$ rows. In the end
we conclude that we obtain $n-1$ independent rows, so
$rank(R)=n-1$ and we are done.

Now we know that $det(R)=Cy^{n-1}$ with $C\neq 0$

\end{proof}

\begin{proposition}

If $n$ is odd then $\xi_{(y,p)}$ is an epimorphism.

\end{proposition}

\begin{proof}

Using lemmas \ref{lemgcd1}, \ref{lemdets1} and \ref{lemdets1} we
may compute $gcd$ of $det_{1,\ldots,n-1}$,
$det_{n+1,\ldots,2n-1}$, $det_{n+1,2,\ldots,n-1}$ and
$det_{n+2,2,\ldots,n-1}$.

\end{proof}

\begin{proposition}

If $n$ is even and $p$ is odd, then $\xi_{(y,p)}$ is an epimorphism.

\end{proposition}

\begin{proof}

It remains to see which is $gcd(2,p-(n^2-nk-n)y$ for all
$k=1,\ldots,n-1$. But $2$ divide $(n^2-nk-n)=n(n-k-1)$ and not $p$
so it cannot divide $p-(n^2-nk-n)y$. We can conclude that
$gcd(2,p-(n^2-nk-n)y=1.$

\end{proof}

\begin{proposition}

If $n$ is even and $p$ is even, then $\xi_{(y,p)}$ is not an epimorphism.

\end{proposition}

\begin{proof}

By remarks \ref{remgentype1} and \ref{remgentype2} we have in this
case that all entries of the matrix are multiples of $2$ and so
$\xi_{(y,p)}$ is not an epimorphism.

\end{proof}

Recall the definition of  $\xi_{(y,p)}$ from $A(\widetilde{A}_{n-1})$ to $W(\widetilde{A}_{n-1})$, $n>2$:

$$\xi_{(y,p)}(\sigma_i)=[\underbrace{y,\ldots,y}_{i-1},-(n-2)y,0,y,\ldots,y](i,i+1),\quad 1 \leq i< n$$
and
$$\xi_{(y,p)}(\sigma_n)=[p,y,\ldots,y,-(n-2)y-p](1,n).$$

We can now state the main result proved in this section:

\begin{proposition}

Let $y,p\in\mathbb{Z}$ and $gcd(y,p)=1$.

\begin{enumerate}

\item If $n\geq 3$ is odd then $\xi_{(y,p)}$ is an epimorphism.
\item If $n\geq 3$ is even and $p$ is odd then $\xi_{(y,p)}$ is an epimorphism.

\end{enumerate}

\end{proposition}

\subsection{Case l=2}

\begin{proposition}

The morphism $\xi$ it is not an epimorphism for all $x_i\in
\mathbb{Z}.$

\end{proposition}

\begin{proof}

The image by $\xi$ of all elements in
$\phi(\pi_1(p(\mathcal{G}),id)$ is $[0,0,0,0,0,0]$.

\end{proof}

\subsection{Case l=3}

When computing $\phi(\pi_1(p(\mathcal{G}),id))$ we obtain only $3$
distinct elements different from zero:

$$[-x_1-x_2, -x_2-x_3, x_1+x_2, x_2+x_3],$$
$$[-x_2, -x_1-x_2-x_3,x_1+x_2+x_3,x_2 ],$$
$$[x_3, -x_3, x_1, -x_1].$$

\begin{proposition}

The morphism $\xi$ is an epimorphism if and only if
$$det_{1,2,3}=(x_3-x_1)(x_3+x_1)(x_3+x_1+2x_2)=\pm 1.$$

\end{proposition}

So we have 8 epimorphisms, $\xi_{(\pm 1,0,0)}$, $\xi_{(\pm 1,\mp 1
,0)}$, $\xi_{(0,0,\pm 1)}$ and $\xi_{(0,\pm 1,\mp 1)}.$

\subsection{Case l=4}

When computing $\phi(\pi_1(p(\mathcal{G}),id))$ we obtain only $3$
distinct elements different from zero:

$$[-x_3-x_1, -x_2-x_3, x_2+x_3, x_3+x_1],$$
$$[-x_3-x_2-x_1, -x_3, x_3+x_2+x_1, x_3],$$
$$ [-x_1, x_1, x_2, -x_2].$$

\begin{proposition}

The morphism $\xi$ is an epimorphism if and only if
$$det_{1,2,3}=(x_2-x_1)(x_2+x_1)(x_2+x_1+2x_3)=\pm 1.$$

\end{proposition}

So we have 8 epimorphisms, $\xi_{(0,\pm 1,0)}$, $\xi_{(0,\pm 1,\mp
1)}$, $\xi_{(\pm 1,0,0)}$ and $\xi_{(\pm 1,0,\mp 1)}.$

\subsection{Case l=5}

This case derives from the general case so the images of
$\sigma_1,\sigma_2$ and $\sigma_3$ are as in the general case and
the image of $\sigma_n$ is equal to the image of $\sigma_2$. This
means also that we have one less parameter.

When computing $\phi(\pi_1(p(\mathcal{G}),id))$ we obtain only
$15$ distinct elements different from zero, that are displayed in
the following matrix:

$$M=\left [\begin{matrix} -x_{1}-x_{2}-x_{3}-x_{4} &  -3x_{4}-x_{2} &
x_{2}-x_{4} & 5x_{4}+x_{3}+x_{2}+x_{1}  \\   2x_{4} &  2x_{4} &
-2x_{4} & -2x_{4}  \\   -x_{1}-x_{2} &
-4x_{4}-x_{2}-x_{3} & 4x_{4}+x_{2}+x_{1} &  x_{2}+x_{3}  \\
2x_{4} &  -2x_{4} &
-2x_{4} &  2x_{4}  \\   2x_{4} &  -2x_{4} &  2x_{4} &  -2x_{4}  \\
x_{4}-x_{1} &  3x_{4}+x_{1} &  -3x_{4}-x_{3} &  x_{3}-x_{4}  \\
-2x_{4} &  -2x_{4} &  2x_{4} &  2x_{4}  \\
-4x_{4}-x_{1}-x_{2} &  -x_{3}-x_{2} &  x_{2}+x_{1} &
4x_{4}+x_{3}+x_{2}  \\   -5x_{4}-x_{1}-x_{2}-x_{3} &
x_{4}-x_{2} &  3x_{4}+x_{2} &  x_{2}+x_{3}+x_{4}+x_{1}  \\
-2x_{4} &  2x_{4} &  -2x_{4} &  2x_{4}  \\   -2x_{4} & 2x_{4} &
2x_{4} &  -2x_{4}  \\   -x_{4}-x_{1} & x_{1}+x_{4} & -x_{4}-x_{3}
&  x_{4}+x_{3}  \\   -3x_{4}-x_{1} &
x_{1}-x_{4} &  x_{4}-x_{3} &  3x_{4}+x_{3}  \\
-3x_{4}-x_{1}-x_{2}-x_{3} &  -x_{4}-x_{2} &
x_{4}+x_{2} &  3x_{4}+x_{3}+x_{2}+x_{1}  \\
-x_{1}-2x_{4}-x_{2} &  -2x_{4}-x_{2}-x_{3} & x_{2}+x_{1}+2x_{4} &
2x_{4}+x_{3}+x_{2}
\end{matrix} \right ]$$

corresponding the line $i$ with the image of the generator $g_i$.

Define, $det_{i_1,i_2,i_3}$ has the determinant of the sub matrix of $M$ formed by 3 distinct lines $i_1,i_2$ and $i_3$ without the last column. Computing, with MAPLE for instance, all possible $det_{i_1,i_2,i_3}$ we conclude that.

\begin{proposition}

If $gcd(x_1,\ldots,x_4)\neq 1$ then $\xi$ is not an epimorphism.

\end{proposition}

\begin{proof}

It is sufficient to notice that in this case all $det_{i_1,i_2,i_3}$ are multiples of $gcd(x_1,\ldots,x_4)$ for all $i_j\in\{1,\ldots,15\}$, $j\in\{1,2,3\}$ and $i_{j_1}\neq i_{j_2}$ if $j_1\neq j_2$.

\end{proof}

Suppose from now on that $gcd(x_1,\ldots,x_4)=1$.

\begin{proposition}

If $gcd(x_1,x_3,x_4)=k>1$  then $\xi$ is not an epimorphism.

\end{proposition}

\begin{proof}

Notice that in this case all
$det_{i_1,i_2,i_3}$ are multiples of $k$.

\end{proof}

\begin{proposition}

If $x_1+x_3$ is even then $\xi$ is not an epimorphism.

\end{proposition}

\begin{proof}

Notice that in this case all
$det_{i_1,i_2,i_3}$ are even.

\end{proof}

Let us define $p=x_1+x_3+2x_4$ and $q=x_1-x_3$.

Now if we rewrite all $det_{i_1,i_2,i_3}$ using the above substitution we have the following results:

\begin{proposition}

If one of the following inequalities hold

\begin{enumerate}
\item $gcd(x_4,p)=k_1>1;$
\item $gcd(x_4,q)=k_2>1;$
\item $gcd(x_4,p+2x_2)=k_3>1.$
\end{enumerate}

then $\xi$ is not an epimorphism.

\end{proposition}

\begin{proof}

If inequality $j$ holds then all $det_{i_1,i_2,i_3}$ are multiples
of $k_j>1$.

\end{proof}

\begin{proposition}

If $x_1+x_3$ is odd, $gcd(x_1,x_3,x_4)=1$ and
$gcd(x_4,p)=gcd(x_4,q)=gcd(x_4,p+2x_2)=1$, then $\xi$ is an
epimorphism.

\end{proposition}

\begin{proof}

We start by seing that if $x_1+x_3$ is odd then $p$, $q$ and
$p+2x_2$ are also odd. Compute the $gcd$ between $det_{1,2,5}$ and
$det_{13,14,15}$.

\end{proof}

\subsection{Case l=6}

When computing $\phi(\pi_1(p(\mathcal{G}),id))$ we obtain only $3$
distinct elements different from zero:

$$[x_1+x_2, -x_1-x_2, x_1+x_3, -x_1-x_3], $$
$$[x_2, x_3, -x_2, -x_3], $$
$$[x_1+x_2+x_3, -x_1, x_1, -x_1-x_2-x_3]$$

\begin{proposition}

The morphism $\xi$ is an epimorphism if and only if
$$det_{1,2,3}=(x_3-x_2)(x_3+x_2)(x_3+x_2+2x_1)=\pm 1.$$

\end{proposition}

So we have 8 epimorphisms, $\xi_{(0,\pm 1,0)}$, $\xi_{(\mp 1,0,\pm
1)}$, $\xi_{(0,0,\pm 1)}$ and $\xi_{(\pm 1,\mp 1,0)}.$

\subsection{Case l=7}

When computing $\phi(\pi_1(p(\mathcal{G}),id))$ we obtain only $3$
distinct elements different from zero:

$$[x_1+x_2, -x_1-x_2, -x_1-x_3, x_3+x_1], $$
$$[x_1+x_2+x_3, -x_1,-x_1-x_2-x_3, x_1], $$
$$ [x_2, x_3, -x_3, -x_2].$$

\begin{proposition}

The morphism $\xi$ is an epimorphism if and only if
$$det_{1,2,3}=(x_3-x_2)(x_3+x_2)(x_3+x_2+2x_1)=\pm 1.$$

\end{proposition}

So we have 8 epimorphisms, $\xi_{(0,\pm 1,0)}$, $\xi_{(\mp 1,0,\pm
1)}$, $\xi_{(0,0,\pm 1)}$ and $\xi_{(\pm 1,\mp 1,0)}.$

\section{Classes of epimorphisms}

In this section we will see which among the previously founded
epimorphisms are actually different modulus an automorphism of
$W(\widetilde{A}_{n-1})$, $n>2$.  We will proceed by analyzing each nontrivial case
introduced in section 7. Note that case 2 has already been
eliminated.

Excepting the case $l=1$, for all other cases we will compute the
image of $\xi(\sigma_i)$ by all graph automorphisms of
$W(\widetilde{A}_{n-1})).$

The sequence of graph automorphisms will be
$id,\gamma,\rho,\rho\gamma,\rho^2,\rho^2\gamma,\rho^3,\rho^3\gamma
$, where $\rho$ and $\gamma$ are, respectively,  the rotation and
the reflection,  introduced in section 6. Note that this group is
the dihedral group of order $8$.

\subsection{Case $l=1$}

\begin{proposition} Let $(y_1,p_1),(y_2,p_2)\in \mathbb{Z}^2$ be different and
$\xi_{(y_1,p_1)}$, $\xi_{(y_2,p_2)}$ be two epimorphisms from
$A(\widetilde{A}_{n-1})$ to $W(\widetilde{A}_{n-1})$. It does not
exist an automorphism $\phi$ of $W(\widetilde{A}_{n-1})$ such that
$$\xi_{y_2,p_2}=\xi_{y_1,p_1}\circ \phi.$$
\end{proposition}

\begin{proof}

Suppose that exist a $\phi$ in the above conditions. We may assume
that its restriction to $S_n$ is, modulo graph conjugation in
$S_n$, the identity. Otherwise we had a new class of epimorphism
of $A(\widetilde{A}_{n-1})$ to $S_n$ which different from $\xi_1$. We
have that, for

$$\xi(\sigma_i)=[\underbrace{y,\ldots,y}_{i-1},-(n-2)y-x_i,x_i,y,\ldots,y](i,i+1)$$

$$\phi(\xi(\sigma_i))=[u_1,\ldots,u_n](i,i+1).$$

So remains to see how does it behaves when restricted to
$\xi(Ker(p \circ \xi))$. Let $g=[Q_1(y,p),\ldots,Q_n(y,p)]$, in
this case being $\phi$ an inner automorphism, a simple
computation, shows that

$$\phi(g)=[Q_{s(1)}(y,p),\ldots,Q_{s(n)}(y,p)]$$
for some permutation $s\in S_n.$ So modulo inner by graph
automorphisms of $A(\widetilde{A}_{n-1})$ we must have
$(y_1,p_1)=(y_2,p_2)$.

\end{proof}

Now we will see how the inner by graph automorphisms of
$W(\widetilde{A}_{n-1})$ acts on $\xi_{(y,p)}$, for $gcd(y,p)=1$.

\begin{lemma}\label{lemmarhoin}

Let  $y,p\in \mathbb{Z}$ and $gcd(y,p)=1$. Then we have:

$$\rho((i,n))=[1,\underbrace{0,\ldots,0}_{i-1},-1,0,\ldots,0](1,i+1)=w_1w_{i+1}^{-1}(1,i+1), \text{ for } 1 \leq i<n,$$

$$\rho((1,i))=(2,i+1)=s_1(1,i+1)s_1, \text{ for } 1< i \leq n-1,$$

\end{lemma}

\begin{proof}

We can write the permutation $(i,n)$ in terms of the generators $s_i$ of $W(\widetilde{A}_{n-1})$. So $(i,n)=s_{i}\cdots s_{n-1}s_{n-2}\cdots s_i$, hence, by definition, $\rho((i,n))=s_{i+1}\cdots s_{n}s_{n-1}\cdots s_{i+1}.$ We simplify to obtain $\rho((i,n))=[1,\underbrace{0,\ldots,0}_{i-1},-1,0,\ldots,0]s_{i+1}\cdots s_{n-1}(1,n)s_{n-1}\cdots s_{i+1}=[1,\underbrace{0,\ldots,0}_{i-1},-1,0,\ldots,0](1,i).$
The proof of the second equality is very similar. We have $(1,i)=s_{1}\cdots s_{i-1}s_{i-2}\cdots s_1$, $i\leq n$,  and, by definition
$\rho((1,i))=s_{2}\cdots s_{i}s_{i-1}\cdots s_2=(2,i+1)$, $i \leq n-1$.

\end{proof}

\begin{proposition}\label{proprhoxi}

Let $n>2$, $1\leq i\leq n$, $y,p\in \mathbb{Z}$ and $gcd(y,p)=1$. The rotation $\rho$ acts on $\xi_{(y,p)}$ in the following way:

$$\rho(\xi_{(y,p)}(\sigma_i))=\xi_{(y,p)}(\sigma_{i+1}) \text{ for }1\leq i\leq n-2;$$

$$\rho(\xi_{(y,p)}(\sigma_{n-1}))=w_1^{1-p}\xi_{(y,p)}(\sigma_{n});$$

$$\rho(\xi_{(y,p)}(\sigma_n))=w_1^{1-p}w_2^{p-1}\xi_{(y,p)}(\sigma_1).$$

\end{proposition}

\begin{proof}

This results of simple computations using proposition \ref{proprho1}.

\end{proof}

\begin{proposition}\label{proprhok}

Let $k>1$, $n>2$, $1\leq i\leq n$, $y,p\in \mathbb{Z}$ and $gcd(y,p)=1$ We have

$$\rho^k(\xi_{(y,p)}(\sigma_i))=\xi_{(y,p)}(\sigma_{i+k\text{ mod }n})\text{ for } i+k-1\text{ mod }n\neq 0,n-1 \text{ and }i\neq n;$$

$$\rho^k(\xi_{(y,p)}(\sigma_i))=w_1^{1-p}\xi_{(y,p)}(\sigma_{n})\text{ for } i+k-1\text{ mod }n = n-1;\text{ and }i\neq n$$

$$\rho^k(\xi_{(y,p)}(\sigma_i))=w_2^{1-p}w_1^{p-1}\xi_{(y,p)}(\sigma_{1})\text{ for } i+k-1\text{ mod }n = 0\text{ and }i\neq n;$$

$$\rho^k(\xi_{(y,p)}(\sigma_n))=w_kw_{k+1}^{-1}\xi_{(y,p)}(\sigma_{k})\text{ for } k-1\text{ mod }n \neq n-1;$$

$$\rho^{n-1}(\xi_{(y,p)}(\sigma_n))=w_{n-1}\xi_{(y,p)}(\sigma_{n-1}).$$

\end{proposition}

\begin{proof}

This proof is a direct use of Proposition \ref{proprho1}.

\end{proof}

\begin{remark}

Note that, $\rho^n(\xi_{(y,p)}(\sigma_i))=\xi_{(y,p)}(\sigma_i)$ for $1 \leq i\leq n.$

\end{remark}

\begin{proposition}\label{propgamma}

In the previous conditions we have:

$$\gamma(\xi_{(y,p)}(\sigma_1))=s_{1}\xi_{(-y,p)}(\sigma_{1})s_{1};$$

$$\gamma(\xi_{(y,p)}(\sigma_2))=w_1^{-p}(1,n)\xi_{(-y,p)}(\sigma_n)s_n;$$

$$\gamma(\xi_{(y,p)}(\sigma_n))=w_2^{1-p}w_3^{p-1}s_2\xi_{(-y,p)}(\sigma_2)s_2.$$

$$\gamma(\xi_{(y,p)}(\sigma_i))=s_{n-i+2}\xi_{(-y,p)}(\sigma_{n-i+2})s_{n-i+2}\text{
for }i\neq 1,2,n;$$

\end{proposition}

\begin{proof}

This is a direct consequence of proposition \ref{propsym}.

\end{proof}

\begin{proposition}

Let $k>0$, $n>2$, $1\leq i\leq n$, $y,p\in \mathbb{Z}$ and $gcd(y,p)=1$ We have

\begin{enumerate}

\item $\gamma\rho^k(\xi_{(y,p)}(\sigma_i))=w_1^{-p}(1,n)\xi_{(-y,p)}(\sigma_n)s_n \text{ for } i+k-1\text{ mod }n=1 \text{ and }i\neq n;$

\item $\gamma\rho^k(\xi_{(y,p)}(\sigma_i))=\xi_{(y,p)}(\sigma_{n+3-(i+k\text{ mod }n)}) \text{ for } i+k-1\text{ mod }n \neq 0,1,n-1 \text{ and }i\neq n;$

\item $\gamma\rho^k(\xi_{(y,p)}(\sigma_i))=w_3^{p-1}s_2\xi_{(-y,p)}(\sigma_2)s_2\text{ for } i+k-1\text{ mod }n = n-1;\text{ and }i\neq n$

\item $\gamma\rho^k(\xi_{(y,p)}(\sigma_i))=w_1^{1-p}w_2^{p-1}s_1\xi_{(-y,p)}(\sigma_{1})s_1 \text{ for } i+k-1\text{ mod }n = 0\text{ and }i\neq n;$

\item $\gamma\rho^2(\xi_{(y,p)}(\sigma_n))=w_2^{-1}w_1^{-p}(1,n)\xi_{(-y,p)}(\sigma_n)s_n \text{ for } k-1\text{ mod }n=1;$

\item $\gamma\rho^3(\xi_{(y,p)}(\sigma_n))=w_4s_{n-1}\xi_{(-y,p)}(\sigma_{n-1})s_{n-1}\text{ for } k-1\text{ mod }n=2;$

\item $\gamma\rho^{k}(\xi_{(y,p)}(\sigma_n))=w_{k+1}w_{k}^{-1}s_{n+3-k}\xi_{(-y,p)}(\sigma_{n+3-k})s_{n+3-k} \text{ for } k-1\text{ mod }n \neq 1,2,n-1;$

\item $\gamma\rho^{n-1}(\xi_{(y,p)}(\sigma_n))=w_3s_3\xi_{(-y,p)}(\sigma_3)s_3.$

\end{enumerate}

\end{proposition}

\begin{proof}

This proof can be done by following the cases of Proposition \ref{proprhok} one by one and using proposition \ref{propsym}.

\end{proof}

We will deal now with the standard epimorphism $\mu$ and $\xi_{(0,-1)}.$ Notice that $\rho(\mu(\sigma_i))=\rho(\xi_{(0,-1)}(\sigma_i))$ for $i \neq n.$

\begin{lemma}\label{lemmarhok01}

Let $s'_n=\xi_{(0,-1)}(\sigma_n).$ Then $\rho^k(s'_n)\neq \rho^k(s_n).$

\end{lemma}

\begin{proof}

We can suppose that $k<n$, because $\rho^n=id$. So

Using Proposition \ref{proprhok} we have, for $k<n-1$:
$$\rho^k(s'_n)=\rho^k(\xi_{(0,-1)}(\sigma_n))=w_kw_{k+1}^{-1}\xi_{(0,-1)}(\sigma_k)=$$

$$=w_kw_{k+1}^{-1}\xi_{(0,1)}(\sigma_k) \neq \xi_{(0,1)}(\sigma_k)=\mu(\sigma_k)=\rho^k(\xi_{(0,1)}(\sigma_n)).$$

\end{proof}

\begin{lemma}\label{lemmagamma01}

Let $n>2$, $k>1$, $s'_n=\xi_{(0,-1)}(\sigma_n).$ Then $\gamma(s'_n)\neq \gamma(s_n).$

\end{lemma}

\begin{proof}

By Proposition \ref{propgamma} we have:

$$\gamma(s'_n)=\gamma(\xi_{(0,-1)}(\sigma_n))=w_2^{2}w_3^{-2}s_2\xi_{(0,-1)}(\sigma_2)s_2=$$
$$=w_2^{2}w_3^{-2}s_2\xi_{(0,1)}(\sigma_2)s_2=w_2^{2}w_3^{-2}\gamma(\xi_{(0,1)}(\sigma_n))= \neq \gamma(\xi_{(0,1)}(\sigma_n))=\gamma(s_n).$$

\end{proof}

\begin{lemma}\label{lemmagammarhok01}

Let $n>2$, $k>1$, $s'_n=\xi_{(0,-1)}(\sigma_n).$ Then $\gamma\rho^k(s'_n)\neq \gamma\rho^k(s_n).$

\end{lemma}

\begin{proof}

We can suppose, again,  that $k<n$. So

Using Proposition \ref{proprhok} we have, for $k<n-1$:
$$\gamma\rho^k(s'_n)=\gamma\rho^k(\xi_{(0,-1)}(\sigma_n))=\gamma (w_kw_{k+1}^{-1}\xi_{(0,-1)}(\sigma_k))=$$

$$=\gamma (w_kw_{k+1}^{-1})\gamma(\xi_{(0,1)}(\sigma_k))=\gamma (w_kw_{k+1}^{-1})\gamma(s_k)$$

If $k=1$ then

$$\gamma\rho(s'_n)=\gamma (w_1w_{2}^{-1})\gamma(s_1)=w_1w_2^{-1}s_1 \neq s_1=\gamma(s_1)=\gamma\rho(s_n).$$

If $k=2$ then

$$\gamma\rho^2(s'_n)=\gamma (w_2w_{3}^{-1})\gamma(s_2)=w_1^{-1}s_n \neq s_n=\gamma(s_2)=\gamma\rho^2(s_n).$$

If $k=3$ then

$$\gamma\rho^3(s'_n)=\gamma (w_3w_{4}^{-1})\gamma(s_3)=w_{n-1} s_{n-1} \neq s_{n-1}=\gamma(s_3)=\gamma\rho^3(s_n).$$

For $k>3$ then

$$\gamma\rho^k(s'_n)=\gamma (w_kw_{k+1}^{-1})\gamma(s_k)=w_{n+2-k}w_{n+3-k}^{-1} s_{n+3-k} \neq s_{n+3-k}=\gamma(s_k)=\gamma\rho^k(s_n).$$

\end{proof}

\begin{proposition}

Let $n>2$, then the morphisms $\mu=\xi_{(0,1)}$ and $\xi_{(0,-1)}$ are different up to automorphisms of $W(\widetilde{A}_{n-1})$.

\end{proposition}

\begin{proof}

Lemmas \ref{lemmarhok01}, \ref{lemmagamma01} and \ref{lemmagammarhok01} show to us that there are no automorphism $\psi$ of $W(\widetilde{A}_{n-1})$ such that $$\psi(\xi_{(0,-1)}(\sigma_n))=\xi_{(0,1)}(\sigma_n)=\mu(\sigma_n).$$

\end{proof}

Finally we state:

\begin{proposition}

If $(y,p)\neq (y',p')$ then the classes, up to automorphisms of $W(\widetilde{A}_{n-1})$,  $\xi_{(y,p)}$ and $\xi_{y',p'}$ are different.

\end{proposition}

\begin{proof}

For a fixed pair $(y,p)$, using the results of Propositions \ref{proprhoxi}, \ref{proprhok} and \ref{propgamma}, we conclude that each automorphism, $\psi$, of $W(\widetilde{A}_{n-1})$ sends $(\xi_{(y,p)}(\sigma_1),\ldots,\xi_{(y,p)}(\sigma_n))$ to a different list $(\psi(\xi_{(y,p)}(\sigma_1)),\ldots,\psi(\xi_{(y,p)}(\sigma_n)))$. The same occurs if we change the pair $(y,p)$. We never obtain  a repeated list.

\end{proof}

\subsection{Case $l=3,4,6,7$}

We present in this section the result of extensive computations using MAPLE. We exemplify with case $l=3$. The cases $l=4,6,7$ are similar. So in case $l=3$, for the epimorphism $\xi_{(1,0,0)}$ we have:

\footnotesize

\noindent\begin{tabular}{|c|cccc|} \hline
& $\xi_{(1,0,0)}(\sigma_1)$ & $\xi_{(1,0,0)}(\sigma_2)$ & $\xi_{(1,0,0)}(\sigma_3)$ & $\xi_{(1,0,0)}(\sigma_4)$ \\ \hline
$id$&[[0,-1,1,0],[[1,2,3,4]]]&[[-1,0,1,0],[[1,3,4,2]]]&[[0,-1,1,0],[[1,2,3,4]]]&[[0,0,1,-1],[[1,2,4,3]]]
\\\hline $\gamma$&[[1,1,0,-2],[[1,4,3,2]]]&[[1,1,0,-2],[[1,2,4
,3]]]&[[1,1,0,-2],[[1,4,3,2]]]&[[0,1,0,-1],[[1,3,4,2]]]
\\\hline $\rho$&[[1,-1,-1,1],[[1,2,3,4]]]&[[1,-1,-1,1],[[1,3
,2,4]]]&[[1,-1,-1,1],[[1,2,3,4]]]&[[0,0,0,0],[[1,4,2,3]]]
\\\hline $\rho\gamma$&[[-1,1,1,-1],[[1,4,3,2]]]&[[-1,1,1,-1],[[1,4
,2,3]]]&[[-1,1,1,-1],[[1,4,3,2]]]&[[0,0,0,0],[[1,3,2,4]]]
\\\hline $\rho^2$&[[2,0,-1,-1],[[1,2,3,4]]]&[[2,0,-1,-1],[[1,2
,4,3]]]&[[2,0,-1,-1],[[1,2,3,4]]]&[[1,0,-1,0],[[1,3,4,2]]]
\\\hline $\rho^2\gamma$&[[0,-1,1,0],[[1,4,3,2]]]&[[0,-1,0,1],[[1,3,4
,2]]]&[[0,-1,1,0],[[1,4,3,2]]]&[[1,-1,0,0],[[1,2,4,3]]]
\\\hline $\rho^3$&[[0,1,0,-1],[[1,2,3,4]]]&[[0,2,0,-2],[[1,4,2
,3]]]&[[0,1,0,-1],[[1,2,3,4]]]&[[1,1,-1,-1],[[1,3,2,4]]]
\\\hline $\rho^3\gamma$&[[1,0,-1,0],[[1,4,3,2]]]&[[2,0,-2,0],[[1,3,2
,4]]]&[[1,0,-1,0],[[1,4,3,2]]]&[[1,1,-1,-1],[[1,4,2,3]]]\\ \hline
\end{tabular}

\normalsize

For the epimorphism $\xi_{(-1,0,0)}$ we have:

\footnotesize

\noindent\begin{tabular}{|c|cccc|} \hline

& $\xi_{(-1,0,0)}(\sigma_1)$ & $\xi_{(-1,0,0)}(\sigma_2)$ & $\xi_{(-1,0,0)}(\sigma_3)$ & $\xi_{(-1,0,0)}(\sigma_4)$ \\ \hline
$id$&[[0,1,-1,0],[[1,2,3,4]]]&[[1,0,-1,0],[[1,3,4,2]]]&[[0,1,-1,0],[[1,2,3,4]]]&[[0,0,-1,1],[[1,2,4,3]]]
\\\hline$\gamma$&[[-1,1,0,0],[[1,4,3,2]]]&[[1,-1,0,0],[[1,2,4
,3]]]&[[-1,1,0,0],[[1,4,3,2]]]&[[0,1,-2,1],[[1,3,4,2]]]
\\\hline$\rho$&[[1,-1,1,-1],[[1,2,3,4]]]&[[1,1,-1,-1],[[1,3
,2,4]]]&[[1,-1,1,-1],[[1,2,3,4]]]&[[2,0,0,-2],[[1,4,2,3]]]
\\\hline$\rho\gamma$&[[1,-1,1,-1],[[1,4,3,2]]]&[[1,1,-1,-1],[[1,4
,2,3]]]&[[1,-1,1,-1],[[1,4,3,2]]]&[[2,0,0,-2],[[1,3,2,4]]]
\\\hline$\rho^2$&[[0,0,-1,1],[[1,2,3,4]]]&[[0,0,1,-1],[[1,2,4
,3]]]&[[0,0,-1,1],[[1,2,3,4]]]&[[-1,2,-1,0],[[1,3,4,2]]]
\\\hline$\rho^2\gamma$&[[0,1,-1,0],[[1,4,3,2]]]&[[0,1,0,-1],[[1,3,4
,2]]]&[[0,1,-1,0],[[1,4,3,2]]]&[[-1,1,0,0],[[1,2,4,3]]]
\\\hline$\rho^3$&[[2,-1,0,-1],[[1,2,3,4]]]&[[0,0,0,0],[[1,4,2
,3]]]&[[2,-1,0,-1],[[1,2,3,4]]]&[[1,-1,1,-1],[[1,3,2,4]]]
\\\hline$\rho^3\gamma$&[[1,0,1,-2],[[1,4,3,2]]]&[[0,0,0,0],[[1,3,2,
4]]]&[[1,0,1,-2],[[1,4,3,2]]]&[[1,-1,1,-1],[[1,4,2,3]]] \\ \hline

\end{tabular}

\normalsize

For the epimorphism $\xi_{(1,-1,0)}$ we have:

\footnotesize

\noindent\begin{tabular}{|c|cccc|} \hline

& $\xi_{(1,-1,0)}(\sigma_1)$ & $\xi_{(1,-1,0)}(\sigma_2)$ & $\xi_{(1,-1,0)}(\sigma_3)$ & $\xi_{(1,-1,0)}(\sigma_4)$ \\ \hline

$id$&[[0,0,1,-1],[[1,2,3,4]]]&[[0,0,1,-1],[[1,3,4,2]]]&[[0,0,1,-1],[[1,2,3,4]]]&[[0,1,0,-1],[[1,2,4,3]]]
\\\hline$\gamma$&[[0,1,1,-2],[[1,4,3,2]]]&[[1,0,1,-2],[[1,2,4
,3]]]&[[0,1,1,-2],[[1,4,3,2]]]&[[-1,1,0,0],[[1,3,4,2]]]
\\\hline$\rho$&[[0,-1,0,1],[[1,2,3,4]]]&[[0,0,-1,1],[[1,3,2
,4]]]&[[0,-1,0,1],[[1,2,3,4]]]&[[0,0,1,-1],[[1,4,2,3]]]
\\\hline$\rho\gamma$&[[-1,0,1,0],[[1,4,3,2]]]&[[-1,1,0,0],[[1,4,2
,3]]]&[[-1,0,1,0],[[1,4,3,2]]]&[[1,-1,0,0],[[1,3,2,4]]]
\\\hline$\rho^2$&[[2,-1,-1,0],[[1,2,3,4]]]&[[2,-1,0,-1],[[1,2
,4,3]]]&[[2,-1,-1,0],[[1,2,3,4]]]&[[0,0,-1,1],[[1,3,4,2]]]
\\\hline$\rho^2\gamma$&[[1,-1,0,0],[[1,4,3,2]]]&[[1,-1,0,0],[[1,3,4
,2]]]&[[1,-1,0,0],[[1,4,3,2]]]&[[1,0,-1,0],[[1,2,4,3]]]
\\\hline$\rho^3$&[[1,1,-1,-1],[[1,2,3,4]]]&[[0,2,-1,-1],[[1,4
,2,3]]]&[[1,1,-1,-1],[[1,2,3,4]]]&[[2,0,-1,-1],[[1,3,2,4]]]
\\\hline$\rho^3\gamma$&[[1,1,-1,-1],[[1,4,3,2]]]&[[1,1,-2,0],[[1,3,
2,4]]]&[[1,1,-1,-1],[[1,4,3,2]]]&[[1,1,0,-2],[[1,4,2,3]]] \\ \hline

\end{tabular}

\normalsize

For the epimorphism $\xi_{(-1,1,0)}$ we have:

\footnotesize

\noindent\begin{tabular}{|c|cccc|} \hline

& $\xi_{(-1,1,0)}(\sigma_1)$ & $\xi_{(-1,1,0)}(\sigma_2)$ & $\xi_{(-1,1,0)}(\sigma_3)$ & $\xi_{(-1,1,0)}(\sigma_4)$ \\ \hline

$id$&[[0,0,-1,1],[[1,2,3,4]]]&[[0,0,-1,1],[[1,3,4,2]]]&[[0,0,-1,1],[[1,2,3,4]]]&[[0,-1,0,1],[[1,2,4,3]]]
\\\hline$\gamma$&[[0,1,-1,0],[[1,4,3,2]]]&[[1,0,-1,0],[[1,2,4
,3]]]&[[0,1,-1,0],[[1,4,3,2]]]&[[1,1,-2,0],[[1,3,4,2]]]
\\\hline$\rho$&[[2,-1,0,-1],[[1,2,3,4]]]&[[2,0,-1,-1],[[1,3
,2,4]]]&[[2,-1,0,-1],[[1,2,3,4]]]&[[2,0,-1,-1],[[1,4,2,3]]]
\\\hline$\rho\gamma$&[[1,0,1,-2],[[1,4,3,2]]]&[[1,1,0,-2],[[1,4,2
,3]]]&[[1,0,1,-2],[[1,4,3,2]]]&[[1,1,0,-2],[[1,3,2,4]]]
\\\hline$\rho^2$&[[0,1,-1,0],[[1,2,3,4]]]&[[0,1,0,-1],[[1,2,4
,3]]]&[[0,1,-1,0],[[1,2,3,4]]]&[[0,2,-1,-1],[[1,3,4,2]]]
\\\hline$\rho^2\gamma$&[[-1,1,0,0],[[1,4,3,2]]]&[[-1,1,0,0],[[1,3,4
,2]]]&[[-1,1,0,0],[[1,4,3,2]]]&[[-1,0,1,0],[[1,2,4,3]]]
\\\hline$\rho^3$&[[1,-1,1,-1],[[1,2,3,4]]]&[[0,0,1,-1],[[1,4,
2,3]]]&[[1,-1,1,-1],[[1,2,3,4]]]&[[0,0,1,-1],[[1,3,2,4]]]
\\\hline$\rho^3\gamma$&[[1,-1,1,-1],[[1,4,3,2]]]&[[1,-1,0,0],[[1,3,
2,4]]]&[[1,-1,1,-1],[[1,4,3,2]]]&[[1,-1,0,0],[[1,4,2,3]]] \\ \hline

\end{tabular}

\normalsize

For the epimorphism $\xi_{(0,0,1)}$ we have:

\footnotesize

\noindent\begin{tabular}{|c|cccc|} \hline

& $\xi_{(0,0,1)}(\sigma_1)$ & $\xi_{(0,0,1)}(\sigma_2)$ & $\xi_{(0,0,1)}(\sigma_3)$ & $\xi_{(0,0,1)}(\sigma_4)$ \\ \hline

$id$&[[1,-1,0,0],[[1,2,3,4]]]&[[0,-1,0,1],[[1,3,4,2]]]&[[1,-1,0,0],[[1,2,3,4]]]&[[1,-1,0,0],[[1,2,4,3]]]
\\\hline$\gamma$&[[1,0,0,-1],[[1,4,3,2]]]&[[2,0,-1,-1],[[1,2,
4,3]]]&[[1,0,0,-1],[[1,4,3,2]]]&[[1,0,-1,0],[[1,3,4,2]]]
\\\hline$\rho$&[[1,0,-1,0],[[1,2,3,4]]]&[[2,0,-2,0],[[1,3,2
,4]]]&[[1,0,-1,0],[[1,2,3,4]]]&[[1,1,-1,-1],[[1,4,2,3]]]
\\\hline$\rho\gamma$&[[0,1,0,-1],[[1,4,3,2]]]&[[0,2,0,-2],[[1,4,2
,3]]]&[[0,1,0,-1],[[1,4,3,2]]]&[[1,1,-1,-1],[[1,3,2,4]]]
\\\hline$\rho^2$&[[1,0,0,-1],[[1,2,3,4]]]&[[1,1,0,-2],[[1,2,4
,3]]]&[[1,0,0,-1],[[1,2,3,4]]]&[[0,1,0,-1],[[1,3,4,2]]]
\\\hline$\rho^2\gamma$&[[0,0,1,-1],[[1,4,3,2]]]&[[-1,0,1,0],[[1,3,4
,2]]]&[[0,0,1,-1],[[1,4,3,2]]]&[[0,0,1,-1],[[1,2,4,3]]]
\\\hline$\rho^3$&[[0,0,0,0],[[1,2,3,4]]]&[[-1,1,1,-1],[[1,4,2
,3]]]&[[0,0,0,0],[[1,2,3,4]]]&[[0,0,0,0],[[1,3,2,4]]]
\\\hline$\rho^3\gamma$&[[0,0,0,0],[[1,4,3,2]]]&[[1,-1,-1,1],[[1,3,2
,4]]]&[[0,0,0,0],[[1,4,3,2]]]&[[0,0,0,0],[[1,4,2,3]]] \\ \hline

\end{tabular}

\normalsize

For the epimorphism $\xi_{(0,0,-1)}$ we have:

\footnotesize

\noindent\begin{tabular}{|c|cccc|} \hline

& $\xi_{(0,0,-1)}(\sigma_1)$ & $\xi_{(0,0,-1)}(\sigma_2)$ & $\xi_{(0,0,-1)}(\sigma_3)$ & $\xi_{(0,0,-1)}(\sigma_4)$ \\ \hline

$id$&[[-1,1,0,0],[[1,2,3,4]]]&[[0,1,0,-1],[[1,3,4,2]]]&[[-1,1,0,0],[[1,2,3,4]]]&[[-1,1,0,0],[[1,2,4,3]]]
\\\hline$\gamma$&[[-1,2,0,-1],[[1,4,3,2]]]&[[0,0,1,-1],[[1,2,
4,3]]]&[[-1,2,0,-1],[[1,4,3,2]]]&[[-1,2,-1,0],[[1,3,4,2]]]
\\\hline$\rho$&[[1,-2,1,0],[[1,2,3,4]]]&[[0,0,0,0],[[1,3,2,
4]]]&[[1,-2,1,0],[[1,2,3,4]]]&[[1,-1,1,-1],[[1,4,2,3]]]
\\\hline$\rho\gamma$&[[0,-1,2,-1],[[1,4,3,2]]]&[[0,0,0,0],[[1,4,2
,3]]]&[[0,-1,2,-1],[[1,4,3,2]]]&[[1,-1,1,-1],[[1,3,2,4]]]
\\\hline$\rho^2$&[[1,0,-2,1],[[1,2,3,4]]]&[[1,-1,0,0],[[1,2,4
,3]]]&[[1,0,-2,1],[[1,2,3,4]]]&[[0,1,-2,1],[[1,3,4,2]]]
\\\hline$\rho^2\gamma$&[[0,0,-1,1],[[1,4,3,2]]]&[[1,0,-1,0],[[1,3,4
,2]]]&[[0,0,-1,1],[[1,4,3,2]]]&[[0,0,-1,1],[[1,2,4,3]]]
\\\hline$\rho^3$&[[2,0,0,-2],[[1,2,3,4]]]&[[1,1,-1,-1],[[1,4,
2,3]]]&[[2,0,0,-2],[[1,2,3,4]]]&[[2,0,0,-2],[[1,3,2,4]]]
\\\hline$\rho^3\gamma$&[[2,0,0,-2],[[1,4,3,2]]]&[[1,1,-1,-1],[[1,3,
2,4]]]&[[2,0,0,-2],[[1,4,3,2]]]&[[2,0,0,-2],[[1,4,2,3]]] \\ \hline

\end{tabular}

\normalsize

For the epimorphism $\xi_{(0,1,-1)}$ we have:

\footnotesize

\noindent\begin{tabular}{|c|cccc|} \hline

& $\xi_{(0,1,-1)}(\sigma_1)$ & $\xi_{(0,1,-1)}(\sigma_2)$ & $\xi_{(0,1,-1)}(\sigma_3)$ & $\xi_{(0,1,-1)}(\sigma_4)$ \\ \hline

$id$&[[1,0,0,-1],[[1,2,3,4]]]&[[1,-1,0,0],[[1,3,4,2]]]&[[1,0,0,-1],[[1,2,3,4]]]&[[1,0,-1,0],[[1,2,4,3]]]
\\\hline$\gamma$&[[0,0,1,-1],[[1,4,3,2]]]&[[2,-1,0,-1],[[1,2,
4,3]]]&[[0,0,1,-1],[[1,4,3,2]]]&[[0,0,-1,1],[[1,3,4,2]]]
\\\hline$\rho$&[[0,0,0,0],[[1,2,3,4]]]&[[1,1,-2,0],[[1,3,2,
4]]]&[[0,0,0,0],[[1,2,3,4]]]&[[1,1,0,-2],[[1,4,2,3]]]
\\\hline$\rho\gamma$&[[0,0,0,0],[[1,4,3,2]]]&[[0,2,-1,-1],[[1,4,2
,3]]]&[[0,0,0,0],[[1,4,3,2]]]&[[2,0,-1,-1],[[1,3,2,4]]]
\\\hline$\rho^2$&[[1,-1,0,0],[[1,2,3,4]]]&[[1,0,1,-2],[[1,2,4
,3]]]&[[1,-1,0,0],[[1,2,3,4]]]&[[-1,1,0,0],[[1,3,4,2]]]
\\\hline$\rho^2\gamma$&[[1,0,0,-1],[[1,4,3,2]]]&[[0,0,1,-1],[[1,3,4
,2]]]&[[1,0,0,-1],[[1,4,3,2]]]&[[0,1,0,-1],[[1,2,4,3]]]
\\\hline$\rho^3$&[[1,0,-1,0],[[1,2,3,4]]]&[[-1,1,0,0],[[1,4,2
,3]]]&[[1,0,-1,0],[[1,2,3,4]]]&[[1,-1,0,0],[[1,3,2,4]]]
\\\hline$\rho^3\gamma$&[[0,1,0,-1],[[1,4,3,2]]]&[[0,0,-1,1],[[1,3,2
,4]]]&[[0,1,0,-1],[[1,4,3,2]]]&[[0,0,1,-1],[[1,4,2,3]]] \\ \hline

\end{tabular}

\normalsize

For the epimorphism $\xi_{(0,-1,1)}$ we have:

\footnotesize

\noindent\begin{tabular}{|c|cccc|} \hline

& $\xi_{(0,-1,1)}(\sigma_1)$ & $\xi_{(0,-1,1)}(\sigma_2)$ & $\xi_{(0,-1,1)}(\sigma_3)$ & $\xi_{(0,-1,1)}(\sigma_4)$ \\ \hline

$id$&[[1,0,0,-1],[[1,2,3,4]]]&[[1,-1,0,0],[[1,3,4,2]]]&[[1,0,0,-1],[[1,2,3,4]]]&[[1,0,-1,0],[[1,2,4,3]]]
\\\hline$\gamma$&[[0,0,1,-1],[[1,4,3,2]]]&[[2,-1,0,-1],[[1,2,
4,3]]]&[[0,0,1,-1],[[1,4,3,2]]]&[[0,0,-1,1],[[1,3,4,2]]]
\\\hline$\rho$&[[0,0,0,0],[[1,2,3,4]]]&[[1,1,-2,0],[[1,3,2,
4]]]&[[0,0,0,0],[[1,2,3,4]]]&[[1,1,0,-2],[[1,4,2,3]]]
\\\hline$\rho\gamma$&[[0,0,0,0],[[1,4,3,2]]]&[[0,2,-1,-1],[[1,4,2
,3]]]&[[0,0,0,0],[[1,4,3,2]]]&[[2,0,-1,-1],[[1,3,2,4]]]
\\\hline$\rho^2$&[[1,-1,0,0],[[1,2,3,4]]]&[[1,0,1,-2],[[1,2,4
,3]]]&[[1,-1,0,0],[[1,2,3,4]]]&[[-1,1,0,0],[[1,3,4,2]]]
\\\hline$\rho^2\gamma$&[[1,0,0,-1],[[1,4,3,2]]]&[[0,0,1,-1],[[1,3,4
,2]]]&[[1,0,0,-1],[[1,4,3,2]]]&[[0,1,0,-1],[[1,2,4,3]]]
\\\hline$\rho^3$&[[1,0,-1,0],[[1,2,3,4]]]&[[-1,1,0,0],[[1,4,2
,3]]]&[[1,0,-1,0],[[1,2,3,4]]]&[[1,-1,0,0],[[1,3,2,4]]]
\\\hline$\rho^3\gamma$&[[0,1,0,-1],[[1,4,3,2]]]&[[0,0,-1,1],[[1,3,2
,4]]]&[[0,1,0,-1],[[1,4,3,2]]]&[[0,0,1,-1],[[1,4,2,3]]] \\ \hline

\end{tabular}

\normalsize

\normalsize

\subsection{Case $l=5$}

This is a different case. For the epimorphism $\xi_{(x_1,x_2,x_3,x_4)}$ we have:

\footnotesize

\noindent\begin{tabular}{|l|ll|} \hline

& $\xi_{(x_1,x_2,x_3,x_4)}(\sigma_1)$ & $\xi_{(x_1,x_2,x_3,x_4)}(\sigma_2)$ \\ \hline

$id$ &  [[-$x_1$ - 2 $x_3$, $x_1$, $x_3$, $x_3$], [[1, 2]]] &

        [[$x_3$, $x_4$, -2 $x_3$ - $x_4$, $x_3$], [[2, 3]]] \\ \hline

$\gamma$ &  [[-$x_1$, $x_1$ + 2 $x_3$, -$x_3$, -$x_3$], [[1, 2]]] &

        [[-$x_4$ + 1, -$x_3$, -$x_3$, 2 $x_3$ + $x_4$ - 1], [[1, 4]]] \\ \hline

$\rho$ &  [[$x_3$, -$x_1$ - 2 $x_3$, $x_1$, $x_3$], [[2, 3]]] &

        [[$x_3$, $x_3$, $x_4$, -2 $x_3$ - $x_4$], [[3, 4]]] \\ \hline

$\rho\gamma$ &  [[$x_1$ + 2 $x_3$ + 1, -$x_3$, -$x_3$, -$x_1$ - 1], [[1, 4]]] &

        [[-$x_3$, -$x_3$, 2 $x_3$ + $x_4$, -$x_4$], [[3, 4]]] \\ \hline

$\rho^2$ &  [[$x_3$, $x_3$, -$x_1$ - 2 $x_3$, $x_1$], [[3, 4]]] &

        [[-2 $x_3$ - $x_4$ + 1, $x_3$, $x_3$, $x_4$ - 1], [[1, 4]]] \\ \hline

$\rho^2\gamma$ &  [[-$x_3$, -$x_3$, -$x_1$, $x_1$ + 2 $x_3$], [[3, 4]]] &

        [[-$x_3$, 2 $x_3$ + $x_4$, -$x_4$, -$x_3$], [[2, 3]]] \\ \hline

$\rho^3$ &  [[$x_1$ + 1, $x_3$, $x_3$, -$x_1$ - 2 $x_3$ - 1], [[1, 4]]] &

        [[$x_4$, -2 $x_3$ - $x_4$, $x_3$, $x_3$], [[1, 2]]] \\ \hline

$\rho^3\gamma$ &  [[-$x_3$, -$x_1$, $x_1$ + 2 $x_3$, -$x_3$], [[2, 3]]] &

        [[2 $x_3$ + $x_4$, -$x_4$, -$x_3$, -$x_3$], [[1, 2]]] \\ \hline

\end{tabular}

\noindent\begin{tabular}{|l|ll|} \hline

& $\xi_{(x_1,x_2,x_3,x_4)}(\sigma_2)$ & $\xi_{(x_1,x_2,x_3,x_4)}(\sigma_3)$ \\ \hline

$id$ &        [[$x_3$, $x_3$, -2 $x_3$ - $x_2$, $x_2$], [[3, 4]]] &

        [[$x_3$, $x_4$, -2 $x_3$ - $x_4$, $x_3$], [[2, 3]]] \\ \hline

$\gamma$ &        [[-$x_3$, -$x_3$, -$x_2$, 2 $x_3$ + $x_2$], [[3, 4]]] &

        [[-$x_4$ + 1, -$x_3$, -$x_3$, 2 $x_3$ + $x_4$ - 1], [[1, 4]]] \\ \hline

$\rho$ &        [[$x_2$ + 1, $x_3$, $x_3$, -2 $x_3$ - $x_2$ - 1], [[1, 4]]] &

        [[$x_3$, $x_3$, $x_4$, -2 $x_3$ - $x_4$], [[3, 4]]] \\ \hline

$\rho\gamma$ &        [[-$x_3$, -$x_2$, 2 $x_3$ + $x_2$, -$x_3$], [[2, 3]]] &

        [[-$x_3$, -$x_3$, 2 $x_3$ + $x_4$, -$x_4$], [[3, 4]]] \\ \hline

$\rho^2$ &        [[-2 $x_3$ - $x_2$, $x_2$, $x_3$, $x_3$], [[1, 2]]] &

        [[-2 $x_3$ - $x_4$ + 1, $x_3$, $x_3$, $x_4$ - 1], [[1, 4]]] \\ \hline

$\rho^2\gamma$ &        [[-$x_2$, 2 $x_3$ + $x_2$, -$x_3$, -$x_3$], [[1, 2]]] &

        [[-$x_3$, 2 $x_3$ + $x_4$, -$x_4$, -$x_3$], [[2, 3]]] \\ \hline

$\rho^3$ &        [[$x_3$, -2 $x_3$ - $x_2$, $x_2$, $x_3$], [[2, 3]]] &

        [[$x_4$, -2 $x_3$ - $x_4$, $x_3$, $x_3$], [[1, 2]]] \\ \hline

$\rho^3\gamma$ &        [[2 $x_3$ + $x_2$ + 1, -$x_3$, -$x_3$, -$x_2$ - 1], [[1, 4]]] &

        [[2 $x_3$ + $x_4$, -$x_4$, -$x_3$, -$x_3$], [[1, 2]]] \\ \hline

\end{tabular}

\normalsize

\normalsize

\subsection{Conclusion}

In the previous computations, and tables resulting from them, we did not find a repeated line. We also checked with all the tables resulting from the cases $l=4,6,7.$ This means that all epimorphisms presented in this section are different up to automorphisms of $W(\widetilde{A}_{n})$ with $n>1$.

\section{The case $A(\widetilde{A}_{1})$}

In this section we deal with the special case $A(\widetilde{A}_{1})$. We state and prove the following result:

\begin{theorem}\label{maintheoremA1}

The representatives of the classes of epimorphisms from
$A(\widetilde{A}_{1})$ to its Coxeter group
$W(\widetilde{A}_{1})$, are $\xi_1^{w,w'}$ and $\xi_2^{w',w}$ defined by:

$$\left\{\begin{array}{c}
     \xi_1^{w,w'}(\sigma_1)=w\\
     \xi_1^{w,w'}(\sigma_2)=w'
  \end{array}\right.$$

where $w=prod_{q'}(s_1s_2)$ with $q'$ odd, $|w'|=2$ and

$$\left\{\begin{array}{c}
     \xi_2^{w',w} (\sigma_1)=w'\\
     \xi_2^{w',w} (\sigma_2)=w
  \end{array}\right.$$

where $w=prod_q(s_1s_2)$ with $q$ odd and $w'=s_1s_2$.

\end{theorem}

Recall that all the automorphisms of $W(\widetilde{A}_{1})$ are inner by graph (see  \cite{Franzsen}).  Let  $\rho$ denote the automorphism of $W(\widetilde{A}_{1})$ the sends $s_1$ to $s_2$ and $s_2$ to $s_1$.

We will state a few lemmas before.

\begin{lemma}\label{lemmalengthgeq3}

Let $\xi$, from $A(\widetilde{A}_{1})$ to $W(\widetilde{A}_{1})$, be a morphism such that

$$\left\{\begin{array}{c}
     \xi(\sigma_1)=w\\
     \xi(\sigma_2)=w'
  \end{array}\right.$$

with $|w|,|w'|\geq 3.$ Then $\xi$ is not an epimorphism.

\end{lemma}

\begin{proof}

Suppose that
$$\xi(\sigma_1)=w=s_1w_1$$ where $w_1=prod_l(s_2,s_1)$, for some $l\geq 1$. Note that we can still suppose in general that $\xi(\sigma_1)$ starts by $s_1$ up to $\rho$. Now suppose that $$\xi(\sigma_2)=w'$$ where $w'=prod_q(s_2,s_1)$, for some $q > 2$. We will see later that assuming that $w'$ starts by $s_2$ is not a real restriction. Then

\begin{enumerate}

\item If $l$ is even and $q$ is odd: in this case there are no cancelations beside the trivial ones and $|\xi(\sigma_i)\xi(\sigma_j)|=|\xi(\sigma_i)||\xi(\sigma_j)>1$, for $i,j\in\{1,2\}$ and $\xi$ is not an epimorphism.
\item\label{case2eq2} If $l$ is even, $q$ is even, $q=l$: Suppose that there ia a word $\omega\in A(\widetilde{A}_{1})$ such that $\xi(\omega)=\sigma_2$. Suppose that $\omega$ has minimal length, $|\omega|$ (in the generators and their inverses), among all words such that $\xi(\omega)=\sigma_2$ and that $|\omega|\geq 3$. If $|\omega|=3$ we have the following possibilities for $\omega$:

\begin{enumerate}

\item\label{cond31} $\omega=\sigma_2^{-1}\sigma_1\sigma_2$, then $|\xi(\omega)|=3q+1>1.$
\item\label{cond32} $\omega=\sigma_2\sigma_1\sigma_2^{-1}$, then $|\xi(\omega)|=q-1>1.$
\item $\omega=\sigma_1\sigma_2\sigma_1$, then $\xi(\omega)=\xi(\sigma_2^{-1})$ and $|\omega|$ is not minimal.
\item $\omega=\sigma_1\sigma_2^{-1}\sigma_1$, then $\xi(\omega)=\xi(\sigma_2)$ and $|\omega|$ is not minimal.

\end{enumerate}

So $\omega$ cannot have length $3$. Suppose that $|\omega|=4$ then

\begin{enumerate}

\item\label{cond41} $\omega=\sigma_2^{-1}\sigma_1\sigma_2\sigma_2$, then $|\xi(\omega)|=4q+1>q+2.$
\item\label{cond42} $\omega=\sigma_2\sigma_1\sigma_2^{-1}\sigma_2^{-1}$, then $|\xi(\omega)|=2q-1>q+2.$
\item $\omega=\sigma_2^{-1}\sigma_1\sigma_2\sigma_1$, then $\xi(\omega)=\xi(\sigma_2^{-2})$  and $|\omega|$ is not minimal.

\end{enumerate}

Now suppose, has our induction hypothesis, that $\omega=\omega_3\omega_1$, $|\omega_3\omega_1|>q+2$ for $|\omega_3|=3$ and $|\omega_1|=k>1$. Let $\omega_2\in A(\widetilde{A}_{1})$ be such that $|\omega_2|=k+1$. So $\omega_2=\omega_1t$ with $|\omega_1|=k$ and $t\in \left\{\sigma_1, \sigma_2, \sigma_2^{-1} \right\}$. We have that $|\xi(\omega_1)|>q+2$ hence for any $t$ the length of $|\xi(\omega_1t)|>q+2-(q+1)>1$, being $q+1$ the biggest amount of cancelations possible.

We conclude that $\xi$ is not an epimorphism.

\item\label{case2neq2} If $l$ is even, $q$ is even, $q\neq l$ and $q\neq 2$: This is just like the previous case the proof being similar. this is not an epimorphism.
\item\label{case2q2} If $l$ is even and $q=2$: We have $\xi(\sigma_1\sigma_2^{-\frac{l}{2}})=s_1$ and $\xi(\sigma_2^{\frac{l}{2}+1}\sigma_1)=s_2$ and we have an epimorphism.
\item If $l$ is odd and $q$ is even: Then $|\xi(\sigma_i)\xi(\sigma_j)|$ are even, for $i,j\in\{1,2\}$ and $\xi$ is not an epimorphism.

\item If $l$ is odd, $q$ is odd and $q=l$: This is case \ref{case2eq2} in which we change the roles of $\sigma_1$ and $\sigma_2$.

\item If $l$ is odd, $q$ is odd, $q \neq l$ and $l \neq 1$: This is, again, a previous case \ref{case2neq2} in which we change the roles of $\sigma_1$ and $\sigma_2$.

\item If $q$ is odd and $l=1$: Finally this is case \ref{case2q2} in which we change the roles of $\sigma_1$ and $\sigma_2$ and we have an epimorphism.
\end{enumerate}

The case where $w'=prod_q(s_1,s_2)$ is the same. We just have to take in consideration of making the products in the reversed order. We obtain one of the previous cases.

\end{proof}

\begin{lemma}\label{lemmalength1geq3}

Let $\xi$, from $A(\widetilde{A}_{1})$ to $W(\widetilde{A}_{1})$, be a morphism such that

$$\left\{\begin{array}{c}
     \xi(\sigma_1)=w\\
     \xi(\sigma_2)=w'
  \end{array}\right.$$

with $|w|=1$ and $|w'|\geq 3.$ Then one of the following hold:

\begin{itemize}

\item $\xi$ is equal to $\mu$ up to an automorphism of $W(\widetilde{A}_{1})$;

\item $\xi$ is not an epimorphism.

\end{itemize}

\end{lemma}

\begin{proof}

Let us suppose that $w=s_1$. We will proceed by induction on the length of $w'$. Suppose now that $|w'|=3$. If $w'=s_1s_2s_1$ then $\xi$ is in the same class  of $\mu$ up to conjugation by $s_1$. Let $w'=s_2s_1s_2$. The proof is similar to the proof of Lemma \ref{lemmalengthgeq3} to show that we cannot have a word of length $2$ as image of $\xi.$

\end{proof}

\begin{lemma}\label{lemmaareepi}

The morphisms from
$A(\widetilde{A}_{1})$ to its Coxeter group
$W(\widetilde{A}_{1})$, defined by:

$$\left\{\begin{array}{c}
     \xi_1^{w,w'}(\sigma_1)=w\\
     \xi_1^{w,w'}(\sigma_2)=w'
  \end{array}\right.$$

where $w=prod_{q'}(s_1s_2)$ with $q'$ odd, $|w'|=2$ and

$$\left\{\begin{array}{c}
     \xi_2^{w',w} (\sigma_1)=w'\\
     \xi_2^{w',w} (\sigma_2)=w
  \end{array}\right.$$

where $w=prod_q(s_1s_2)$ with $q$ odd and $w'=s_1s_2$ are epimorphisms. This epimorphisms are different from the standard epimorphism $\mu.$

\end{lemma}

\begin{proof}

Suppose that $w'=s_1s_2$, then $\xi_1^{w,w'}(\sigma_1(\sigma_2)^{\frac{q'-1}{2}})=s_1$ and $\xi_1^{w,w'}(\sigma_1(\sigma_2)^{\frac{q'+1}{2}})=s_2$. This same argument is used in every case showing that this are in fact epimorphisms.

To see that $\xi_1^{w,w'}$ is different from $\mu$ notice that supposing that $w'=s_1s_2$ we have
$$\xi_1^{w,w'}(\sigma_2^2)=s_1s_2s_1s_2$$ and $$\mu(\sigma_2^2)=id.$$  There is no automorphism of $W(\widetilde{A}_{1})$ sending $\xi(\sigma_2^2)$ to $id$.

\end{proof}

Consider the following  epimorphisms from
$A(\widetilde{A}_{1})$ to its Coxeter group
$W(\widetilde{A}_{1})$, defined by:

$$\left\{\begin{array}{c}
     \xi_1^{w,w'}(\sigma_1)=w\\
     \xi_1^{w,w'}(\sigma_2)=w'
  \end{array}\right.$$

where $w=prod_{q'}(s_1s_2)$ with $q'$ odd, $|w'|=2$ and

$$\left\{\begin{array}{c}
     \xi_2^{w',w} (\sigma_1)=s_1s_2\\
     \xi_2^{w',w} (\sigma_2)=w
  \end{array}\right.$$

where $w=prod_q(s_1s_2)$ with $q$ odd.

\begin{lemma}\label{lemmaarediff}

Let $\xi_1^{w,w'}$,$\xi_1^{w'',w'}$,$\xi_2^{w',w'''}$ and $\xi_2^{w',w}$ be four epimorphisms defined as above. Then, if $i\neq j$ or $w\neq w''$ or $w'\neq w'''$ then it does not exist an automorphism $\psi$ of $W(\widetilde{A}_{1})$ such that $\xi_i^{w,w'}=\psi(\xi_j^{w'',w'''})$

\end{lemma}

\begin{proof}

Now to see which of this epimorphisms are the same up to an automorphism of $W(\widetilde{A}_{1})$. Let $\xi_1^w$,$\xi_1^{w'}$ be epimorphisms such that $w=prod_q(s_2s_1) \neq w'=prod_{q'}(s_2s_1)$ and $q'>q$. Suppose that there is an automorphism of $W(\widetilde{A}_{1})$, $\psi$, such that  $\psi(\xi_1^w)=\xi_1^{w'}$. In this case if we compute $$\psi(\xi_1^w)(\sigma_1)=\psi(s_1prod_q(s_2s_1))=s_1prod_{q'}(s_2s_1)=\xi_1^{w'}(\sigma_1)$$ and $$\psi(\xi_1^w)(\sigma_2)=\psi(s_2s_1)=s_2s_1=\xi_1^{w'}(\sigma_2).$$ Hence $\psi(s_1prod_{q}(s_2s_1))=\psi(s_1)prod_q(s_2s_1)$ if $q$ is even and $\psi(s_1prod_{q}(s_2s_1))=\psi(s_1)prod_{q-1}(s_2s_1)\psi(s_2)$ if $q$ is odd. Suppose that $q$ is even. In this case we have that $\psi(s_1)=s_1prod_{q-q'}(s_2s_1)$ and $\psi(s_2)=\psi(s_2s_1)\psi(s_1)=s_1prod_{q'-q-2}(s_2s_1)$. If $q'$ is odd then $|\psi(s_1^2)|=2|s_1prod_{q-q'}(s_2s_1)|$ and $\psi$ is not an automorphism. So $q'$ must be even. This is, again, not possible because in this case the length of the image of any word by $\psi$ is even.

Suppose now that $q$ is odd. We have $$\psi(s_1)prod_{q-1}(s_2s_1)\psi(s_2)=s_1prod_{q'}(s_2s_1)$$ which is equivalent to $$\psi(s_2)\psi(s_1)prod_{q-1}(s_2s_1)\psi(s_2)=\psi(s_2)s_1prod_{q'}(s_2s_1),$$ $$prod_{q+1}(s_2s_1)\psi(s_2)=\psi(s_2)s_1prod_{q'}(s_2s_1)$$
and $$prod_{q+1}(s_2s_1)=\psi(s_2)s_1prod_{q'}(s_2s_1)\psi(s_2).$$
If $q'$ is even this means that $2|\psi(s_2)|=q'-q$ which is false because $q'$ is even, $q$ is odd and $q'-q$ is even. So $\psi$ is not an automorphism.

This means that if $w\neq w'$ then the class of $\xi_1^w$ is different from the class of $\xi_1^{w'}$.

The case of the epimorphisms of type $\xi_2^w$ is analogous the the previous one.

Now we prove that the class of $\xi_2^w$ is different from the class of $\xi_1^{w'}$.
Suppose that there is an automorphism $\psi$ such that:

$$\psi(\xi_1^{w'})=\xi_2^w.$$

We have:

$$\psi(\xi_1^{w'}(\sigma_1))=\psi(s_1w')=s_1s_2=\xi_2^w (\sigma_1)$$  and
$$\psi(\xi_1^{w'}(\sigma_2))=\psi(s_2s_1)=w=\xi_2^w (\sigma_2).$$

Notice that $w=prod_q(s_2s_1)$ with $q$ odd. So $w=w^{-1}$ and $\psi(s_2s_1s_2s_1)=id$ which is impossible (this equality implies that the group $W(\widetilde{A}_{1})$ is finite). So there is not such automorphism and we are done.

\end{proof}

We summarize the proof of the Main Theorem of this section.

\begin{proof}[Proof of theorem \ref{maintheoremA1}]

Lemmas \ref{lemmalengthgeq3} and \ref{lemmalength1geq3} state the shape of the candidates to be epimorphism. Lemma \ref{lemmaareepi} shows that the morphisms that are not covered by Lemmas \ref{lemmalengthgeq3} and \ref{lemmalength1geq3} are epimorphisms. Finally Lemma \ref{lemmaarediff} shows that the epimorphisms of Lemma \ref{lemmaareepi} are different up to an automorphism of $W(\widetilde{A}_{1})$.

\end{proof}

\section{Main result}

So finally we present, in an condensed form, the main result
already proved:

\begin{theorem}

The representatives of the classes of epimorphisms from
$A(\widetilde{A}_{n-1})$ to its Coxeter group
$W(\widetilde{A}_{n-1})$, for $n>1$, are:

\begin{enumerate}

\item $\left\{\begin{array}{c}
     \xi_1^{w,w'}(\sigma_1)=w\\
     \xi_1^{w,w'}(\sigma_2)=w'
  \end{array}\right.$ where $w=prod_{q'}(s_1s_2)$ with $q'$ odd, $|w'|=2$ and

\item $\left\{\begin{array}{c}
     \xi_2^{w',w} (\sigma_1)=s_1s_2\\
     \xi_2^{w',w} (\sigma_2)=w
  \end{array}\right.$ where $w=prod_q(s_1s_2)$ with $q$ odd.

  \item $(\xi_1)_{(0,\pm 1)}$ for all $n \geq 3$.
  \item $(\xi_1)_{(y,p)}$ for $gcd(y,p)=1$ and $n\geq 3$ odd.
  \item $(\xi_1)_{(y,p)}$ for $gcd(y,p)=1$, $n\geq 3$ even and $p$ odd.

  \item $(\xi_3)_{(x_1,x_2,x_3)}$ for $(x_1,x_2,x_3)\in \{ (\pm 1,0,0),(\pm 1,\mp 1,0),(0,0,\pm 1),(0,\pm 1 ,\mp 1)\}$.

  \item $(\xi_4)_{(x_1,x_2,x_3)}$ for $(x_1,x_2,x_3)\in \{ (0,\pm 1,0),(0,\pm 1,\mp 1),(\pm 1,0,0),(\pm 1,0 ,\mp 1)\}$.

  \item $(\xi_5)_{(x_1,\ldots,x_4)}$ for $x_1+x_3$ odd, $gcd(x_1,x_3,x_4)=gcd(x_4,x_1+x_3+2x_4)=gcd(x_4,x_1-x_3)=gcd(x_4,x_1+x_2+x_3+2x_4)=1.$

  \item $(\xi_6)_{(x_1,x_2,x_3)}$ for $(x_1,x_2,x_3)\in \{ (0,\pm 1,0),(\pm 1,\mp 1,0),(0,0,\pm 1),(\pm 1,0 ,\mp 1)\}.$
  \item $(\xi_7)_{(x_1,x_2,x_3)}$ for $(x_1,x_2,x_3)\in \{ (0,\pm 1,0),(\pm 1,\mp 1,0),(0,0,\pm 1),(\pm 1,0 ,\mp 1)\}.$
\end{enumerate}

  Where the epimorphisms $\xi_k$ are the ones introduced in the
  $k^{th}$ subsection of section 3.

\end{theorem}


\begin{thebibliography}{A}
\bibitem[A]{artin}  E. Artin. \textit{Braids and permutations}. Annals of Math.
\textbf{48}, 3, (1947) 643-649.

\bibitem[A1]{artin1}  E. Artin. \textit{Theory of Braids}. Annals of Math.
\textbf{48} (1946) 101-126.

\bibitem[B]{birman}  J. Birman. \textit{Braids, Links and Mapping Class
Groups}. Annals of Math. Studdies \textbf{82}, Princeton
University Press, (1973).

\bibitem[Bou]{bourbaki} N. Bourbaki, \textit{Groupes et alg\`{e}bre de Lie, Chapitres 4,5 et 6}, Hermann, Paris, 1968.

\bibitem[CP]{pariscohen} A. M. Cohen, L. Paris, \textit{On a Theorem of Artin}, J. Group Theory, 6 (2003), 421-441.

\bibitem[FP]{francoparis} N. Franco, L. Paris "On a theorem of Artin, II, Journal of Group theory, Volume: 9, 731 - 751, (2006).

\bibitem[F]{Franzsen} W. N. Franzsen, \textit{Automorphisms of coxeter
groups}, Phd Thesis, University of Sydney,
2001.

\bibitem[RG-S]{ZeCarlos} J. C. Rosales and P. A. Garc\'{i}a-S\'{a}nchez, \textit{Finitely generated commutative monoids}, Nova Science Publishers (1999)
ISBN 978-1560726708.

\end{thebibliography}
\end{document}